\documentclass[a4paper,12pt]{article}
\usepackage[english]{babel}
\usepackage{amssymb,amsmath,amsthm}
\usepackage{enumerate,enumitem}
\usepackage{booktabs} 
\usepackage{caption}
\usepackage{tikz,graphicx,color}
\usepackage{epsfig}
\usepackage{subcaption}
\usepackage{bm}
\usepackage{xurl}
\usepackage{hyperref}
\usepackage[nameinlink]{cleveref}
\usepackage[hmargin=2.3cm,vmargin=2.36cm]{geometry}
\usepackage{enumitem}
\usepackage{authblk}

\usepackage[square,sort,comma,numbers]{natbib}
\theoremstyle{plain}
\newtheorem{thm}{Thm}[section]

\newtheorem{claim}[thm]{Claim}
\newtheorem{theorem}[thm]{Theorem}
\newtheorem{lemma}[thm]{Lemma}
\newtheorem{corollary}[thm]{Corollary}
\newtheorem{proposition}[thm]{Proposition}
\newtheorem{conjecture}[thm]{Conjecture}

\newtheorem{observation}[thm]{Observation}

\numberwithin{equation}{section}

\newlength{\bibitemsep}
\newlength{\bibparskip}
\let\oldthebibliography\thebibliography
\renewcommand\thebibliography[1]{
  \oldthebibliography{#1}
  \setlength{\parskip}{\bibitemsep}
  \setlength{\itemsep}{\bibparskip}
}

\newenvironment{proofofclaim}
  {\begin{proof}[Proof of the claim]}
  {\end{proof}}

\begin{document}
\title{Frustration indices of signed subcubic graphs}

\date{}

\author{Sirui Chen} 

\author{Jiaao Li}

\author{Zhouningxin Wang}

\affil{\small School of Mathematical Sciences and LPMC, Nankai University, Tianjin 300071, China.

Email: sirui.chen@mail.nankai.edu.cn, \{lijiaao, wangzhou\}@nankai.edu.cn.}

\maketitle

\begin{abstract}
The frustration index of a signed graph is defined as the minimum number of negative edges among all switching-equivalent signatures. This can be regarded as a generalization of the classical \textsc{Max-Cut} problem in graphs, as the \textsc{Max-Cut} problem is equivalent to determining the frustration index of signed graphs with all edges being negative signs. 
In this paper, we prove that the frustration index of an $n$-vertex signed connected simple  subcubic graph, other than $(K_4, -)$, is at most $\frac{3n + 2}{8}$, and we characterize the family of signed graphs for which this bound is attained. This bound can be further improved to $\frac{n}{3}$ for signed $2$-edge-connected simple subcubic graphs, with the exceptional signed graphs being characterized. As a corollary, every signed $2$-edge-connected simple cubic  graph on at least $10$ vertices and with $m$ edges has its frustration index at most $\frac{2}{9}m$, where the upper bound is tight as it is achieved by an infinite family of signed cubic graphs.
\end{abstract}


\section{Introduction}

The \textsc{Max-Cut} problem seeks a vertex partition of a graph that maximizes the number of edges between the two parts. For a graph $G$, let $b(G)$ denote the size of its maximum cut, and let $v(G)$ and $e(G)$ denote the number of vertices and edges, respectively. Although the \textsc{Max-Cut} problem is shown to be NP-hard \cite{K1972}, several positive algorithmic results exist under specific constraints. Hadlock \cite{H1975} proved that a maximum cut can be found in polynomial time for planar graphs. Barahona \cite{B1983} further showed that the problem is polynomially solvable for graphs not contractible to $K_5$. Goemans and Williamson \cite{GW1995} proposed a polynomial-time algorithm that finds a cut in any graph $G$ of size at least approximately $0.878 b(G)$, which is best possible. 

Edwards \cite{E1973,E1975} established a fundamental lower bound of $b(G)$ based on its number of vertices and edges: $b(G) \geq \frac{e(G)}{2} + \frac{v(G) - 1}{4}.$ In 1997, J. Koml\'{o}s~\cite{K1997} studied simple graphs $G$ of girth $g$, and provided a sharp upper bound of $\min\{v(G), c \frac{v(G)}{g} \log \frac{2v(G)}{g} \}$, where $c$ is a constant. In 1982, Locke \cite{L1982} showed that for any $k$-regular graph $G$, the number of edges of a maximal $(k-1)$-colorable subgraph is greater than $\frac{k^2-2}{k^2}e(G)$, which implies, in particular, that $b(G) \geq \frac{7}{9}e(G)$ for cubic graphs $G$. Further refinements were obtained in subcubic settings. Hopkins and Staton \cite{HS1982} proved that $b(G) \geq \frac{4}{5}e(G)$ for triangle-free cubic graphs. Bondy and Locke \cite{BL1986} extended the result of Hopkins and Staton to triangle-free subcubic graphs, and Xu and Yu \cite{XY2008} characterized the seven extremal graphs achieving the equality. Zhu \cite{Z2009a} gave a more concise proof of Bondy and Locke’s result,  and later \cite{Z2009b} improved the bound to a tight one: $b(G)\ge \frac{17}{21}e(G)$. For cubic graphs with large girths, Zýka \cite{Z1990} showed that $b(G)\ge \frac{6}{7}e(G)$ when the girth $g$ of $G$ tends to infinity.

An equivalent formulation of the \textsc{Max-Cut} problem can be given in terms of signed graphs and their \emph{frustration indices}. A \emph{signed graph} $(G, \sigma)$ consists of an underlying graph $G$ together with a \emph{signature} $\sigma: E(G) \to \{+, -\}$. The notation $\widehat{G}$ is also used to refer to a signed graph when the signature is clear from the context. The set of negative edges in $(G, \sigma)$ is denoted by $E^-_{(G, \sigma)}$. The \emph{switching} operation at an edge-cut $[X, X^c]$ is to change the signs of all edges in the cut. Two signatures $\sigma$ and $\sigma'$ on $G$ are \emph{switching equivalent} if one can be obtained from the other by switching at some edge-cut. The \emph{frustration index} of $(G, \sigma)$, denoted by $F(G, \sigma)$, is defined to be the minimum number of negative edges in any switching-equivalent signature:
$$F(G, \sigma) = \min \{|E^-_{(G, \sigma')}|: \sigma' \text{~is switching equivalent to~}\sigma\},$$
where the minimum is taken over all switching-equivalent signatures $\sigma'$ of $\sigma$. 
Let $(G, -)$ denote the signed graph with all edges being negative. Noting that in $(G, -)$ all odd cycles have an odd number of negative edges and all even cycles have an even number of negative edges, we may apply a switching at an edge-cut of $G$ with the maximum size, thereby turning all those edges into positive signs. The fundamental relationship between the frustration index and the max-cut follows: $$b(G) = e(G) - F(G, -) \geq e(G) - \max_{\sigma} F(G, \sigma),$$
where the maximum is taken over all arbitrary signatures $\sigma$ of $G$.

Since computing the frustration index can be reduced from the \textsc{Max-Cut} problem where all edges are negative, it is known to be NP-hard. In 1981, J. Akiyama, D. Avis, V. Chv\'{a}tal, and H. Era~\cite{AACE1981} established a non-trivial lower bound for the frustration index of a signed simple graph $(G, \sigma)$; this bound was later refined by P. Solé and T. Zaslavsky in 1994~\cite{SZ1994}. The trivial upper bound $\frac{e(G)}{2}$ was improved by S. Poljak and D. Turz\'{i}k \cite{PT1982} in 1982. Those two results together provide the following:
$$\frac{1}{2}e(G)-\sqrt{\frac{1}{2}\ln{2}}\sqrt{e(G)(v(G)-1)} \leq F(G, \sigma) \leq \frac{1}{2}e(G) - \frac{1}{2} \lceil \frac{1}{2}(v(G)-1) \rceil.$$
 In 2004, D. Král and H. J. Voss \cite{KV2004} proved that for a signed planar graph $(G, \sigma)$, it holds that $F(G, \sigma) \leq 2\nu(G, \sigma)$, where $\nu(G, \sigma)$ denotes the maximum number of edge-disjoint negative cycles in $(G, \sigma)$. For further questions and recent developments in the theory of signed graphs, we refer the reader to the survey \cite{Z1998}.

In this paper, we study the frustration indices of signed subcubic graphs and improve upon the existing results. In all figures, negative edges are depicted as dashed red lines, positive edges as solid blue lines, and edges with unspecified signs are shown as solid gray lines. Our first main result is as follows.

\begin{theorem}\label{thm:main}
Every signed $2$-edge-connected simple subcubic graph $(G, \sigma)$, except for the five specific signed graphs $\widehat{\Gamma}_1, \ldots, \widehat{\Gamma}_5$ shown in \Cref{fig:Exception}, satisfies that 
$$F(G, \sigma) \leq \frac{1}{3} 
v(G).$$ Moreover, there exists an infinite family of signed $2$-edge-connected simple subcubic graphs with the frustration index equal to one-third of the number of vertices.
\end{theorem}

\begin{figure}[ht]
    \centering
    \begin{subfigure}[t]{.4\textwidth}
    \centering
    \begin{tikzpicture}
    [scale=.25]
    \draw(0,0) node[circle, draw=black!90, inner sep=0mm, thick, minimum size=2mm] (x_{0}){};
    \foreach \i/\j in {1,2,3}
    {
    \draw[rotate=120*(\i)+90] (0,4) node[circle, draw=black!90, inner sep=0mm, thick, minimum size=2mm] (x_{\i}){};
    }
    
    \foreach \i/\j in {2/3,3/1}
    {
    \draw[line width=1pt, blue] (x_{\i}) -- (x_{\j});
    }
    \foreach \i in {1,2}
    {
    \draw[line width=1pt, blue] (x_{\i}) -- (x_{0});
    }
    \draw[dashed, line width=1pt, red] (x_{1}) -- (x_{2});
    \draw[dashed, line width=1pt, red] (x_{3}) -- (x_{0});
    \end{tikzpicture}
    \caption{$F(\widehat{\Gamma}_1)=\frac{1}{2}v(\widehat{\Gamma}_1)$}
    \label{fig:K4}
    \end{subfigure}
    \begin{subfigure}[t]{.4\textwidth}
    \centering
    \begin{tikzpicture}
    [scale=.25]
    \draw(0,0) node[circle, draw=black!90, inner sep=0mm, thick, minimum size=2mm] (x_{0}){};
    \draw[rotate=90] (0,-2) node[circle, draw=black!90, inner sep=0mm, thick,  minimum size=2mm] (x_{4}){};
    \foreach \i/\j in {1,2,3}
    {
    \draw[rotate=120*(\i)+90] (0,4) node[circle, draw=black!90, inner sep=0mm, thick, minimum size=2mm] (x_{\i}){};
    }
    
    \foreach \i/\j in {2/3,3/1}
    {
    \draw[line width=1pt, blue] (x_{\i}) -- (x_{\j});
    }
    \foreach \i in {1,2}
    {
    \draw[line width=1pt, blue] (x_{\i}) -- (x_{0});
    }
    \draw[dashed, line width=1pt, red] (x_{1}) -- (x_{4});
    \draw[dashed, line width=1pt, red] (x_{3}) -- (x_{0});
    \draw[line width=1pt, blue] (x_{2}) -- (x_{4});
    \end{tikzpicture}
    \caption{$F(\widehat{\Gamma}_2)=\frac{2}{5}v(\widehat{\Gamma}_2)$}
    \label{fig:W}
    \end{subfigure}
  
    \begin{subfigure}[t]{.3\textwidth}
    \centering
    \begin{tikzpicture}
    [>=latex,
    roundnode/.style={circle, draw=black!90, inner sep=0mm, thick, minimum size=2mm},
    scale=0.25] 
    \draw(-1,0) node[circle, draw=black!90, inner sep=0mm, thick, minimum size=2mm] (x_{0}){};
    \draw(1,0) node[circle, draw=black!90, inner sep=0mm, thick, minimum size=2mm] (y_{0}){};
    
    \foreach \i/\j in {1,2,3,4,5,6}
    {
    \draw[rotate=60*(\i)+30] (0,4) node[circle, draw=black!90, inner sep=0mm, thick, minimum size=2mm] (x_{\i}){};
    }
    
    \foreach \i/\j in {1/2,2/3,3/4,4/5,1/6}
    {
    \draw[line width=1pt, blue] (x_{\i}) -- (x_{\j});
    }
    \foreach \i in {2,6}
    {
    \draw[line width=1pt, blue] (x_{\i}) -- (x_{0});
    }
    
    \foreach \i in {3,5}
    {
    \draw[line width=1pt, blue] (x_{\i}) -- (y_{0});
    }
    
    \draw[dashed, line width=1pt, red] (x_{6}) -- (x_{5});
    \draw[dashed, line width=1pt, red] (x_{1}) -- (x_{0});
    \draw[dashed, line width=1pt, red] (x_{4}) -- (y_{0});
    \end{tikzpicture}
    \caption{$F(\widehat{\Gamma}_3)=\frac{3}{8}v(\widehat{\Gamma}_3)$}
    \label{fig:Gamma1}
    \end{subfigure}
    \begin{subfigure}[t]{.3\textwidth}
    \centering
    \begin{tikzpicture}
    [>=latex, roundnode/.style={circle, draw=black!90, inner sep=0mm, thick, minimum size=2mm}, scale=0.25] 
    \draw(-2,0) node[circle, draw=black!90, inner sep=0mm, thick, minimum size=2mm] (x_{0}){};
    \draw(2,0) node[circle, draw=black!90, inner sep=0mm, thick, minimum size=2mm] (y_{0}){};
    \draw(0,1.5) node[circle, draw=black!90, inner sep=0mm, thick, minimum size=2mm] (z_{0}){};

    \foreach \i/\j in {1,2,3,4,5}
    {\draw[rotate=360/5*(\i)] (0,4) node[circle, draw=black!90, inner sep=0mm, thick, minimum size=2mm] (x_{\i}){};}
    
    \foreach \i/\j in {1/2,2/3,4/5,5/1}
    {\draw[line width=1pt, blue] (x_{\i}) -- (x_{\j});}
    \foreach \i in {3,4}
    {\draw[line width=1pt, blue] (x_{\i}) -- (y_{0});}
    \draw[line width=1pt, blue] (x_{1}) -- (x_{0}) -- (z_{0}) -- (y_{0});
    \draw[dashed, line width=1pt, red] (x_{0}) -- (x_{2});
    \draw[dashed, line width=1pt, red] (x_{3}) -- (x_{4});
    \draw[dashed, line width=1pt, red] (x_{5}) -- (z_{0});
    \end{tikzpicture}
    \caption{$F(\widehat{\Gamma}_4)=\frac{3}{8}v(\widehat{\Gamma}_4)$}
    \label{fig:Gamma2}
    \end{subfigure}
    \begin{subfigure}[t]{.3\textwidth}
    \centering
    \begin{tikzpicture}
    [>=latex,
    roundnode/.style={circle, draw=black!90, inner sep=0mm, thick, minimum size=2mm},
    scale=0.6] 
    \foreach \i in {1,2,3,4}
    {	\draw[rotate=-90*(\i-2)+45] (0, 2) node[roundnode] (u_\i){};
    }
    \foreach \i in {1,2,3,4}
    {	\draw[rotate=-90*(\i-2)+45] (0, 0.8) node[roundnode] (v_\i){};
    }
    \foreach \i/\j in {1/2,2/3,3/4}
    {	\draw [line width=1pt, blue] (u_\i) -- (u_\j);
    }
    \foreach \i/\j in {1/2,2/3,4/1}
    {
    \draw [line width=1pt, blue] (v_\i) -- (v_\j);
    }
    \foreach \i in {1,3,4}
    {
    \draw [line width=1pt, blue] (u_\i) -- (v_\i);
    }
    \foreach \i/\j in {4/1}
    {
    \draw [dashed, line width=1pt, red] (u_\i) -- (u_\j);
    }
    \foreach \i/\j in {3/4}
    {
    \draw [dashed, line width=1pt, red] (v_\i) -- (v_\j);
    }
    \foreach \i in {2}
    {
    \draw [dashed, line width=1pt, red] (v_\i) -- (u_\i);
    }
    \end{tikzpicture}
    \caption{$F(\widehat{\Gamma}_5)=\frac{3}{8}v(\widehat{\Gamma}_5)$}
    \label{fig:Cube}
    \end{subfigure}
    \caption{Five exceptional signed subcubic graphs with $F(G, \sigma)>\frac{1}{3}v(G)$}
    \label{fig:Exception}
\end{figure}
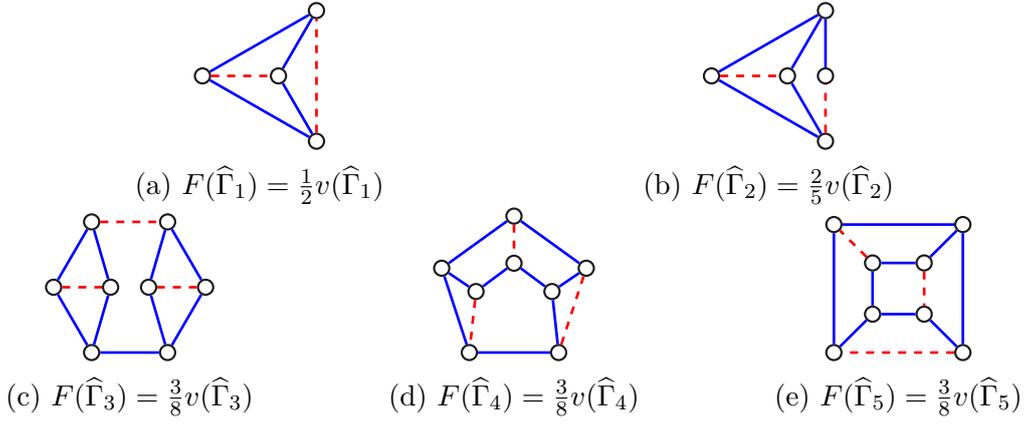

We now present an infinite family of signed $2$-edge-connected simple subcubic graphs $(G_n, \sigma_n)$ with $F(G_n, \sigma_n)=\frac{1}{3} v(G_n)$. Each of $(G_n, \sigma_n)$ is formed by connecting the components $\widehat{F}_1, \widehat{F}_2, \ldots, \widehat F_n$ by adding positive edges $x_iy_{i-1}$ with indices taken modulo n. Here $\widehat F_i$ is one of the following two kinds of signed gadgets: 
(1) a negative triangle with exactly one vertex of degree $2$; 
(2) any $2$-subdivision of $\widehat \Gamma_1$, where we subdivide a negative edge into a negative edge and a positive edge, and subdivide a positive edge into two positive edges. 
See~\Cref{fig:c3-2} and~\Cref{fig:k4sub}. Note that $|E^-_{\widehat F_i}|=\frac{1}{3}v(F_i)$ and thus $F(G_n, \sigma_n)=\frac{1}{3}v(G_n)$.

\begin{figure}[htbp]
\begin{subfigure}[t]{.45\textwidth}
    \centering  
	    \begin{tikzpicture}[>=latex,
		roundnode/.style={circle, draw=black!90, thick, minimum size=2mm, inner sep=0pt},
        squarenode/.style={rectangle, draw=black!90, thick, minimum size=2mm, inner sep=0pt},
        scale=0.8
		]
          \node[roundnode] (a1) at (1,1) {};
          \node[roundnode] (x1) at (0,0) {};
          \node[roundnode] (y1) at (2,0) {};

          \draw[fill=white,line width=0.2pt] (1,1) node[above=1mm] {$a_i$};
          \draw[fill=white,line width=0.2pt] (0,0) node[below=1mm] {$x_i$};
          \draw[fill=white, line width=0.2pt] (2,0) node[below=1mm] {$y_i$};

          \draw[line width=1pt, blue] (a1)--(x1)--(y1);
          \draw[line width=1pt, blue] (x1)--(-1,0);
          \draw[line width=1pt, blue] (y1)--(3,0);
          \draw[dashed, line width=1pt, red] (y1)--(a1);

	    \end{tikzpicture}
	    \caption{The first kind of signed gadget}
    \label{fig:c3-2}  
     \end{subfigure}
     \begin{subfigure}[t]{.45\textwidth}
	    \centering
	     \begin{tikzpicture}[>=latex,
		roundnode/.style={circle, draw=black!90, thick, minimum size=2mm, inner sep=0pt},
        squarenode/.style={rectangle, draw=black!90, thick, minimum size=2mm, inner sep=0pt},
        scale=0.8
		]
          
          \node[roundnode] (a1) at (0.8,0.8) {};
          \node[roundnode] (b1) at (2.4,0.8) {};
          \node[roundnode] (c1) at (0.8,-0.8) {};
          \node[roundnode] (d1) at (2.4,-0.8) {};
          \node[roundnode] (x1) at (0.8,0) {};
          \node[roundnode] (y1) at (2.4,0) {};

          \draw[fill=white,line width=0.2pt] (0.8,0.8) node[above=1mm] {$a_i$};
          \draw[fill=white,line width=0.2pt] (2.4,0.8) node[above=1mm] {$b_i$};          
          \draw[fill=white,line width=0.2pt] (0.8,-0.8) node[below=1mm] {$c_i$};
          \draw[fill=white,line width=0.2pt] (2.4,-0.8) node[below=1mm] {$d_i$};
          \draw[fill=white,line width=0.2pt] (0.9,-0.3) node[left=1mm] {$x_i$};
          \draw[fill=white, line width=0.2pt] (2.3,-0.3) node[right=1mm] {$y_i$};

          \draw[line width=1pt, blue] (a1)--(b1)--(y1)--(d1)--(c1)--(x1)--(a1);
          \draw[dashed, line width=1pt, red] (a1)--(d1);
          \draw[dashed, line width=1pt, red] (c1)--(b1);
          \draw[line width=1pt, blue] (x1)--(0,0);
          \draw[line width=1pt, blue] (y1)--(3.2,0);
	    \end{tikzpicture} 
    \caption{The second kind of signed gadget: a $2$-subdivision of $\widehat \Gamma_1$}  
    \label{fig:k4sub}
    \end{subfigure}
      \caption{A subgraph $\widehat{G_i}$}
    \label{tight} 
\end{figure}
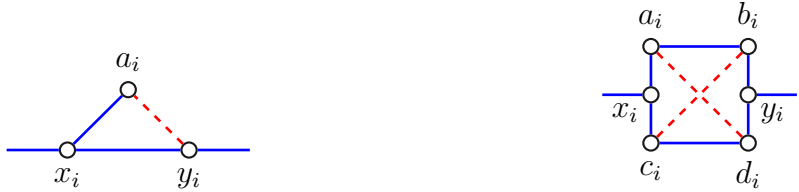

\Cref{thm:main} yields the following corollary for signed cubic graphs, which particularly generalizes the result of Locke~\cite{L1982} and shows that every $2$-edge-connected simple cubic graph $G$ on at least $10$ vertices contains an edge-cut of size at least $\frac{7}{9}e(G)$.

\begin{corollary}
    Every signed $2$-edge-connected simple cubic graph $(G, \sigma)$ on at least $10$ vertices satisfies that $F(G, \sigma)\leq \frac{2}{9}e(G)$. Moreover, the upper bound is tight.
\end{corollary}

For signed subcubic graphs which are not $2$-edge-connected and may contain some cut-edges, we have the following result with the assistance of~\Cref{thm:main}. Each maximal $2$-edge-connected component of a connected graph obtained by removing all cut-edges is called a \emph{block}. A block of a connected graph is called a \emph{leaf-block} if it is connected to the rest of the graph via a single cut-edge. 

\begin{theorem}\label{thm:1ec}
    Every signed connected simple subcubic  graph $(G,\sigma)$ other than $\widehat{\Gamma}_1$ satisfies that $$F(G,\sigma)\le \frac{3v(G)+2}{8}.$$ Moreover, the equality holds if and only if $(G, \sigma)$ is cubic and contains at least one cut-edge, each leaf-block is isomorphic to $\widehat{\Gamma}_2$, and each non-leaf-block is isomorphic to a negative triangle.
\end{theorem}

Two examples are illustrated in~\Cref{fig:3/8-leaf}, showing that the upper bound of~\Cref{thm:1ec} is tight.

\begin{figure}[htbp]
\begin{subfigure}[t]{.45\textwidth}
    \centering  
	    \begin{tikzpicture}[>=latex,
		roundnode/.style={circle, draw=black!90, thick, minimum size=2mm, inner sep=0pt},
        squarenode/.style={rectangle, draw=black!90, thick, minimum size=2mm, inner sep=0pt},
        scale=0.8
		]
        
          \node[roundnode] (x1) at (-1,1) {};
          \node[roundnode] (a1) at (-1,2) {};
          \node[roundnode] (b1) at (-3,2) {};
          \node[roundnode] (c1) at (-3,0) {};
          \node[roundnode] (d1) at (-1,0) {};

          \node[roundnode] (x2) at (1,1) {};
          \node[roundnode] (a2) at (1,2) {};
          \node[roundnode] (b2) at (3,2) {};
          \node[roundnode] (c2) at (3,0) {};
          \node[roundnode] (d2) at (1,0) {};

        \draw[line width=1pt, blue] (x1)--(a1)--(b1)--(c1)--(d1)--(x1);
        \draw[dashed, line width=1pt, red] (a1)--(c1);
        \draw[dashed, line width=1pt, red] (b1)--(d1);
        
        \draw[line width=1pt, blue] (x1)--(x2);
        
        \draw[line width=1pt, blue] (x2)--(a2)--(b2)--(c2)--(d2)--(x2);
        \draw[dashed, line width=1pt, red] (a2)--(c2);
        \draw[dashed, line width=1pt, red] (b2)--(d2);

	    \end{tikzpicture}
	    \caption{$F(G, \sigma)=\frac{3v(G)+2}{8}$}
        \label{fig:leaf-leaf}
     \end{subfigure}
     \begin{subfigure}[t]{.45\textwidth}
	    \centering
	     \begin{tikzpicture}[>=latex,
		roundnode/.style={circle, draw=black!90, thick, minimum size=2mm, inner sep=0pt},
        squarenode/.style={rectangle, draw=black!90, thick, minimum size=2mm, inner sep=0pt},
        scale=0.8
		]
          
          \node[roundnode] (a) at (1,0.4) {};
          \node[roundnode] (b) at (2,0.4) {};
          \node[roundnode] (c) at (3,1) {};
          \node[roundnode] (d) at (3,-0.2) {};
          \node[roundnode] (e) at (4,1.65) {};
          \node[roundnode] (f) at (4,-0.85) {};

          \node[roundnode] (x1) at (-1,1) {};
          \node[roundnode] (y1) at (0,1) {};
          \node[roundnode] (a1) at (-1,1.4) {};
          \node[roundnode] (b1) at (-1.8,1.4) {};
          \node[roundnode] (c1) at (-1.8,0.6) {};
          \node[roundnode] (d1) at (-1,0.6) {};

          \node[roundnode] (x2) at (-1,-0.2) {};
          \node[roundnode] (y2) at (0,-0.2) {};
          \node[roundnode] (a2) at (-1,0.2) {};
          \node[roundnode] (b2) at (-1.8,0.2) {};
          \node[roundnode] (c2) at (-1.8,-0.6){};
          \node[roundnode] (d2) at (-1,-0.6) {};

          \node[roundnode] (x3) at (6,2.3) {};
          \node[roundnode] (y3) at (5,2.3) {};
          \node[roundnode] (a3) at (6,2.7) {};
          \node[roundnode] (b3) at (6.8,2.7) {};
          \node[roundnode] (c3) at (6.8,1.9) {};
          \node[roundnode] (d3) at (6,1.9) {};

          \node[roundnode] (x4) at (6,  1) {};
          \node[roundnode] (y4) at (5,  1) {};
          \node[roundnode] (a4) at (6,  1.4) {};
          \node[roundnode] (b4) at (6.8,1.4) {};
          \node[roundnode] (c4) at (6.8,0.6) {};
          \node[roundnode] (d4) at (6,  0.6) {};

          \node[roundnode] (x5) at (6,  -0.2) {};
          \node[roundnode] (y5) at (5,  -0.2) {};
          \node[roundnode] (a5) at (6,  0.2) {};
          \node[roundnode] (b5) at (6.8,0.2) {};
          \node[roundnode] (c5) at (6.8,-0.6) {};
          \node[roundnode] (d5) at (6,  -0.6) {};

          \node[roundnode] (x6) at (6,  -1.5) {};
          \node[roundnode] (y6) at (5,  -1.5) {};
          \node[roundnode] (a6) at (6,  -1.1) {};
          \node[roundnode] (b6) at (6.8,-1.1) {};
          \node[roundnode] (c6) at (6.8,-1.9) {};
          \node[roundnode] (d6) at (6,  -1.9) {};

        \draw[line width=1pt, blue] (y1)--(x1)--(a1)--(b1)--(c1)--(d1)--(x1);
        \draw[dashed, line width=1pt, red] (a1)--(c1);
        \draw[dashed, line width=1pt, red] (b1)--(d1);
        
        \draw[line width=1pt, blue] (y2)--(x2)--(a2)--(b2)--(c2)--(d2)--(x2);
        \draw[dashed, line width=1pt, red] (a2)--(c2);
        \draw[dashed, line width=1pt, red] (b2)--(d2);
        
        \draw[line width=1pt, blue] (y3)--(x3)--(a3)--(b3)--(c3)--(d3)--(x3);
        \draw[dashed, line width=1pt, red] (a3)--(c3);
        \draw[dashed, line width=1pt, red] (b3)--(d3);
        
        \draw[line width=1pt, blue] (y4)--(x4)--(a4)--(b4)--(c4)--(d4)--(x4);
        \draw[dashed, line width=1pt, red] (a4)--(c4);
        \draw[dashed, line width=1pt, red] (b4)--(d4);
        
        \draw[line width=1pt, blue] (y5)--(x5)--(a5)--(b5)--(c5)--(d5)--(x5);
        \draw[dashed, line width=1pt, red] (a5)--(c5);
        \draw[dashed, line width=1pt, red] (b5)--(d5);
        
        \draw[line width=1pt, blue] (y6)--(x6)--(a6)--(b6)--(c6)--(d6)--(x6);
        \draw[dashed, line width=1pt, red] (a6)--(c6);
        \draw[dashed, line width=1pt, red] (b6)--(d6);
        
          \draw[line width=1pt, blue] (y2)--(y1)--(a)--(b)--(c)--(d);
          \draw[line width=1pt, blue] (c)--(e)--(y3)--(y4);
          \draw[line width=1pt, blue] (d)--(f)--(y5)--(y6);
          \draw[dashed, line width=1pt, red] (a)--(y2);
          \draw[dashed, line width=1pt, red] (b)--(d);
          \draw[dashed, line width=1pt, red] (e)--(y4);
          \draw[dashed, line width=1pt, red] (f)--(y6);
	    \end{tikzpicture} 
    \caption{$\frac{F(G, \sigma)}{v(G)}=\frac{3k+4}{8k+10}$ where $k$ is the number of non-leaf-blocks}  
    \label{fig:k3tree}
    \end{subfigure}
      \caption{Examples illustrating the tightness of~\Cref{thm:1ec}}
    \label{fig:3/8-leaf} 
\end{figure}
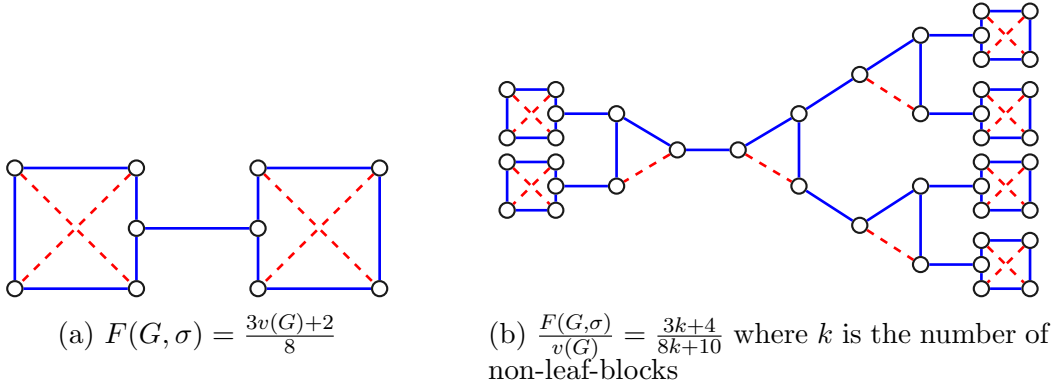 

Note that all the above extremal examples contain certain small cycles. We conjecture that the bounds established above can be further improved by forbidding specific configurations. For example, we propose the following conjecture.

\begin{conjecture}\label{conj3-10}
Every signed connected simple subcubic graph $(G, \sigma)$ with girth at least $5$ satisfies that $F(G, \sigma) \leq \frac{3}{10}v(G)$.
\end{conjecture}

The condition of girth $5$ in~\Cref{conj3-10} is necessary due to the signed graph $\widehat{\Gamma}_5$ in~\Cref{fig:Cube}. Note that the bound $\frac{3}{10}v(G)$ is tight if it is true, as evidenced by a signed Petersen graph with all edges negative, which attains this bound. We leave this conjecture for future research.

\medskip
We now introduce some terminologies in the end of this section.
Let $G$ be a graph. Let $X$ be a vertex subset of $V(G)$ and $u$ be a vertex of $V(G)$. We use $N(u)$ to denote the set of neighbors of $u$ in $G$, and use $d_X(u)$ to denote the number of neighbors of $u$ that are in $X$. For a positive integer $k$, a vertex $u$ is called a \emph{$k$-vertex} if its degree is equal to $k$, that is, $|N(u)|=k$. For any $X\subset V(G)$, let $[X, X^c]$ denote the \emph{edge-cut} between $X$ and $X^c$. An edge-cut $[X, X^c]$ is called a \emph{$k$-edge-cut} if it consists of $k$ edges. A graph $G$ is \emph{$k$-edge-connected} if every edge-cut of $G$ contains at least $k$ edges. 
In this paper, to simplify the notation, we 
write $[x_1,x_2,...,x_n]$ for the edge-cut $[ \{ x_1,x_2,...,x_n  \}, \{ x_1,x_2,...,x_n \}^c ]$.

Let $(G, \sigma)$ be a signed graph. A vertex $x$ of $(G, \sigma)$ is called a \emph{negative} (or \emph{positive}) neighbor of $y$ if the edge $xy$ is negative (or positive) in $(G, \sigma)$. A \emph{digon} is a signed graph on $2$ vertices with two parallel edges, one positive and the other negative.

We say an edge-cut (or simply, a cut) of a signed graph is \emph{unequilibrated} if it contains more negative edges than positive ones, and \emph{equilibrated} if it contains the same number of negative edges and positive edges.

In this paper, we define the \emph{subdivision} operation of any signed graph as follows: to subdivide a negative edge $xy$, delete $xy$, add a new vertex $z$, then add a positive edge $xz$ and a negative edge $yz$; to subdivide a positive edge $xy$, delete $xy$, add a new vertex $z$, then add two positive edges $xz$ and $yz$. 

In the sequel, we adopt the following convention: when constructing a signed graph $(G', \sigma')$ from the original signed graph $(G, \sigma)$, we assume that every unchanged edge inherits its sign from the original signature. 

\medskip
The remainder of the paper is organized as follows: In the next section, we first prove~\Cref{thm:main} under the assumption that a key lemma (\Cref{lem:key}) holds, and then proceed to establish~\Cref{thm:1ec}. In~\Cref{sec:9vertices}, we characterize all exceptional signed subcubic graphs of~\Cref{thm:main} on at most $9$ vertices. Specifically, we explain how we discover the five exceptional signed graphs. In~\Cref{sec:keylemma}, we first establish some properties of the minimum counterexample to~\Cref{thm:main}, and finally provide a proof of~\Cref{lem:key}.

\section[proof]{Frustration indices of signed subcubic graphs}
We start with the following well-known observation in signed graphs.
\begin{observation}\label{lem:bad-cut}
Let $(G,\sigma)$ be a signed graph with $|E^-_{(G, \sigma)}|=F(G,\sigma)$. For any $k$-edge-cut $[X,X^c]$ of $(G, \sigma)$, we have $|E^-_{(G, \sigma)} \cap [X,X^c]|\leq \lfloor \frac{k}{2} \rfloor$.
\end{observation}

It follows that a signed graph $(G, \sigma)$ such that $|E^-_{(G, \sigma)}|=F(G, \sigma)$ cannot have any unequilibrated edge-cut. In particular, at each vertex of a signed subcubic graph $(G, \sigma)$ with $|E^-_{(G,\sigma)}|=F(G, \sigma)$, there is at most one negative edge incident to it. A direct consequence is that for all subcubic graphs $(G, \sigma)$, $F(G, \sigma)\leq \frac{1}{2}v(G)$.

Let $k$ be a positive integer.
Let $\mathcal{T}$ denote the set of signed trees $T$ where $v(T)=2k+1$, and $[V(T), V(T)^c]$ containing at least $k+2$ negative edges. Similarly, let $\mathcal{C}$ denote the set of signed odd cycles $C$ with $v(C)=2k+1$ such that $[V(C), V(C)^c]$ contains at least $k+1$ negative edges. A signed graph is called $\mathcal{T}$-free ($\mathcal{C}$-free, respectively) if it does not contain a tree $T\in \mathcal{T}$ (a cycle $C\in \mathcal{C}$, respectively) as an induced subgraph. Note that in any subcubic graph containing a tree $T$ or a cycle $C$ as an induced subgraph, the cut $[V(T), V(T)^c]$ contains at most $v(T)+2$ edges, while the cut $[V(C), V(C)^c]$ contains at most $v(C)$ edges. This observation leads to the following lemma.

\begin{lemma}\label{lem:tree}
Every signed subcubic graph $(G, \sigma)$ such that $|E^-_{(G, \sigma)}|=F(G, \sigma)$ is $\mathcal{T}$-free and $\mathcal{C}$-free.
\end{lemma}

Next we start to prove~\Cref{thm:main}. Let $(G, \sigma)$ be a minimum counterexample to~\Cref{thm:main} with $v(G)$ being minimized. Thus $(G,\sigma)$ is subcubic, simple, $2$-edge-connected, not isomorphic to any of signed graphs in~\Cref{fig:Exception}, and satisfies that $F(G, \sigma)>\frac{1}{3}v(G)$. Furthermore, among all switching-equivalent signatures, we choose the signature $\sigma$ such that $|E^-_{(G, \sigma)}|=F(G, \sigma).$

We define two subsets $X$ and $Y$ of $V(G)$ with respect to the signature $\sigma$: 
$$X:=\{u\mid u \text{~is not incident to any negative edge in~} (G, \sigma)\} \text{~and ~} Y=V(G)-X.$$ Thus $(X, Y)$ is a bipartition of $V(G)$. Since $E^-_{(G, \sigma)} \subset E(G[Y])$, every vertex of $Y$ has at least one neighbor in $Y$, i.e., $d_X(u)\leq 2$ for every vertex $u\in Y$. Moreover, for each $i\in \{0,1,2,3\}$ and $j\in \{0,1,2\}$, we define that
$$X_i:=\{u\in X\mid d_Y(u)=i\},~\text{~and~}~Y_j:=\{v\in Y\mid d_X(v)=j\}.$$
Note that $X=\bigcup\limits_{k=0}^{3} X_k$ and $Y=\bigcup\limits_{\ell=0}^{2} Y_\ell$. We shall establish the following key lemma, with the proof postponed to~\Cref{sec:keylemma}.

\begin{lemma}\label{lem:key}
In the minimum counterexample $(G, \sigma)$, $|X_3|+|Y_0|\leq |Y_2|$.
\end{lemma}

We first prove~\Cref{thm:main} assuming that~\Cref{lem:key} holds.

\begin{proof}[Proof of \Cref*{thm:main}]
Assume that $(G,\sigma)$ is a minimum counterexample with $v(G)$ being minimized and $|E^-_{(G, \sigma)}|=F(G, \sigma)$. Let $X,Y$, $X_i$ (for $i\in \{0,1,2,3\}$), and $Y_j$ (for $j\in \{0,1,2\}$) be defined as above.
By~\Cref{lem:key}, we have $|X_3|+|Y_0|\leq |Y_2|$, i.e., $|X_3|\leq |Y_2|-|Y_0|$. Note that $|X|+|Y|=v(G)$, $|X_2|=v(G)-|Y|-|X_0|-|X_1|-|X_3|$, and $|Y_1|=|Y|-|Y_0|-|Y_2|$. Note that
\begin{align*}
\sum\limits_{u\in X}d_Y(u) &=\sum\limits_{i=0}^{3}\sum\limits_{u\in X_i}d_Y(u)= 0|X_0|+|X_1|+2|X_2|+3|X_3|\\
&=|X_1|+2(v(G)-|Y|-|X_0|-|X_1|-|X_3|)+3|X_3|,
\end{align*}
and
\begin{align*}
\sum\limits_{v\in Y}d_X(v) 
&= \sum\limits_{i=0}^{2}\sum\limits_{v\in Y_i}d_X(v)=0|Y_0|+|Y_1|+2|Y_2|=|Y|-|Y_0|+|Y_2| \geq |Y|+|X_3|.
\end{align*}
Since $\sum\limits_{u\in X}d_Y(u)=\sum\limits_{v\in Y}d_X(v)$, combining the above two inequalities, $$ |Y|+|X_3|\leq |X_1|+2(v(G)-|Y|-|X_0|-|X_1|-|X_3|)+3|X_3|,$$ solving which we get $|Y|\leq \frac{2}{3}v(G)-\frac{1}{3}(2|X_0|+|X_1|)$. Since $|E^-_{(G, \sigma)}|=F(G, \sigma)$, by~\Cref{lem:bad-cut} $(G, \sigma)$ contains no adjacent negative edges. Thus the number of negative edges of $(G, \sigma)$ is at most half the size of the set $Y$, i.e., $|E^-_{(G, \sigma)}|\leq \frac{1}{2}|Y|$. Hence, $$F(G, \sigma)\leq \frac{1}{3}v(G)-\frac{1}{3}|X_0|-\frac{1}{6}|X_1|,$$ a contradiction. This completes the proof of~\Cref{thm:main}.
\end{proof}

\medskip
Next we shall apply~\Cref{thm:main} to establish~\Cref{thm:1ec}. We first derive a result based on~\Cref{thm:main} that provides a sufficient condition for a signed connected simple subcubic graph $(G, \sigma)$ to have its frustration index at most $\frac{1}{3}v(G)$. 

In the sequel, for every leaf-block $H$, let $e_H$ denote the cut-edge connecting to $H$, and let $v_{H}\in e_H$ denote the vertex in $H$, and $v_{H^c}\in e_H$ denote the vertex not in $H$. 

\begin{proposition}\label{prop:leaf-block}
Let $(G,\sigma)$ be a signed connected simple subcubic graph not isomorphic to any of $\widehat{\Gamma}_1,\ldots, \widehat{\Gamma}_5$. If $(G,\sigma)$ has no leaf-block isomorphic to $\widehat \Gamma_2$, then $F(G,\sigma)\le \frac{1}{3}v(G)$. 
\end{proposition}

\begin{proof}
   Assume that $(G,\sigma)$ is a minimum counterexample with $v(G)$ being minimized. So $F(G,\sigma)>\frac{1}{3}v(G)$. By~\Cref{thm:main}, $G$ must have a cut-edge. Let $H$ be a leaf-block of $(G,\sigma)$ and let $e_H$ be the cut-edge connecting $H$ with $G-H$. Note that $d_H(v_H)\leq 2$ and $d_{G-H}(v_{H^c})\leq 2$. By the assumption, $H$ is not isomorphic to $\widehat \Gamma_2$ and, moreover, as $d_H(v_H)\leq 2$, $H$ is not isomorphic to any of the exceptional signed graphs in~\Cref{fig:Exception}. Hence, applying~\Cref{thm:main} to $(H, \sigma|_H)$, we have $F(H, \sigma|_{H})\le \frac{1}{3}v(H)$. Now we consider the signed subgraph $(G-H, \sigma|_{G-H})$. If $G-H$ is also a block (which is a leaf-block), then as $d_{G-H}(v_{H^c})\leq 2$, by symmetry we have $F(G-H, \sigma|_{G-H})\leq \frac{1}{3}v(G-H)$. Since $e_H$ is a cut-edge, $$F(G, \sigma)=F(H, \sigma|_H)+F(G-H, \sigma|_{G-H}) \leq \frac{1}{3}v(H)+\frac{1}{3}v(G-H)= \frac{1}{3}v(G),$$ a contradiction. Otherwise, $G-H$ contains at least one cut-edge. Since $d_{G-H}(v_{H^c})\le 2$, the newly-built leaf-block of $G-H$ cannot be isomorphic to $\widehat \Gamma_2$. Noting that $F(G-H, \sigma|_{G-H})=F(G, \sigma)-F(H, \sigma|_H)>\frac{1}{3}v(G)-\frac{1}{3}v(H)=\frac{1}{3}v(G-H)$, $G-H$ is a smaller counterexample, a contradiction.
\end{proof}

Now we provide the proof of~\Cref{thm:1ec}.

\begin{proof}[Proof of the first part of \Cref*{thm:1ec}]
Assume to the contrary that $(G, \sigma)$ is a signed connected simple subcubic graph, not isomorphic to $\widehat{\Gamma}_1$, satisfying $F(G,\sigma)> \frac{3v(G)+2}{8}$,
and with $v(G)$ being minimized. As $F(\widehat \Gamma_i,\sigma)< \frac{3v(G)+2}{8}$ for $i\in \{2,3,4,5\}$, by~\Cref{thm:main}, $(G, \sigma)$ contains at least one cut-edge.

We begin by establishing that every leaf-block of $(G,\sigma)$ is isomorphic to $\widehat \Gamma_2$. Assume not and let $H$ be a leaf-block that is not isomorphic to $\widehat \Gamma_2$. Note that since $d_H(v_H)\leq 2$, $H$ is not isomorphic to any of the exceptional signed graphs $\widehat{\Gamma}_i$'s. By~\Cref{thm:main}, $F(H,\sigma|_{H})\le \frac{v(H)}{3}<\frac{3v(H)}{8}$. Since $e_H$ is a cut-edge, we have $$F(G-H, \sigma|_{G-H})=F(G,\sigma)-F(H,\sigma|_{H})>\frac{3v(G)+2}{8}-\frac{3v(H)}{8} =\frac{3v(G-H)+2}{8}.$$ Observing that $d_{G-H}(v_{H^c})\leq 2$, $(G-H, \sigma|_{G-H})$ cannot be isomorphic to $\widehat{\Gamma}_1$, contradicting the minimality of $(G, \sigma)$.

Next, we show that for each leaf-block $H$ of $(G, \sigma)$, $d_G(v_{H^c})=3$. It is easy to see that $d_G(v_{H^c})\ge 2$. If $d_G(v_{H^c})=2$, then we denote the other neighbor of $v_{H^c}$ by $v$, and form a new signed graph by deleting $v_{H^c}$ and adding a positive edge $v_Hv$. The resulting signed graph has the same frustration index as $(G, \sigma)$ but fewer vertices, and not isomorphic to $\widehat \Gamma_1$, contradicting the minimality of $(G, \sigma)$.

We then claim that for each leaf-block $H$ of $(G, \sigma)$, the vertex $v_{H^c}\in e_H$ is in a triangle. Let $x$ and $y$ be the two neighbors of $v_{H^c}$ in $G-H$, and assume to the contrary that $xy\not\in E(G)$. We form a new signed graph $(G', \sigma')$ from $(G, \sigma)$ by deleting $H$ and the vertex $v_{H^c}$, and adding an edge $xy$ with $\sigma'(xy)=\sigma(xv_{H^c})\sigma(yv_{H^c})$. The signed graph $(G', \sigma')$ is simple and connected. Moreover, we have that $$F(G',\sigma')=F(G, \sigma)-F(H, \sigma|_H)>\frac{3v(G)+2}{8}-2=\frac{3(v(G')+6)+2}{8}-2>\frac{3v(G')+2}{8}.$$ Note that $(G',\sigma')$ is not isomorphic to $\widehat \Gamma_1$, as otherwise $(G,\sigma)$ contains $10$ vertices with $F(G,\sigma)=4$, contradicting the assumption of $F(G,\sigma)> \frac{3v(G)+2}{8}$. So $(G',\sigma')$ is a smaller counterexample, a contradiction. 

Moreover, we prove that each leaf-block is connected to a triangle via a cut-edge, and that each such triangle is associated with exactly one leaf-block. Let $H$ be a leaf-block. By the previous claims, it is connected to a triangle $v_{H^c}xy$. Assume to the contrary that at least one of $x$ and $y$ coincides with some $v_{H^c_0}$, say $x=v_{H^c_0}$. If $y$ is a $2$-vertex, then $(G,\sigma)$ is a signed graph on $13$ vertices with $F(G,\sigma)\leq 5$, contradicting the assumption of $F(G,\sigma)> \frac{3v(G)+2}{8}$. So $y$ is a $3$-vertex and let $y'$ be the third neighbor of $y$. In this case, we form a signed graph $(G'', \sigma'')$ by deleting $H_0,v_{H_0^c},x$ and $y$, and adding an edge $v_{H_0^c}y'$ with $\sigma'(v_{H_0^c}y')=\sigma(yy')$. Noting that $(G'',\sigma'')$ is simple and not isomorphic to $\widehat{\Gamma}_1$, we have that $$F(G'',\sigma'')\geq F(G, \sigma)-3> \frac{3v(G)+2}{8}-3=\frac{3(v(G'')+8)+2}{8}-3=  \frac{3v(G'')+2}{8},$$
that is to say, $(G'',\sigma'')$ is a smaller counterexample, a contradiction. 

\smallskip
Let $H_1,\ldots, H_m$ be the leaf-blocks of $G$, and let $(G^*, \sigma^*)$ denote the signed graph $(G-\cup H_i, \sigma|_{G-\cup H_i})$. Note that the triangles corresponding to distinct vertices $v_{H_i}$ are vertex-disjoint in $G$. Assume not, and if there exist two triangles $v_{H_i}xy$ and $v_{H_j}xy$ sharing the same edge $xy$, then $(G, \sigma)$ is a signed graph on $14$ vertices with $F(G,\sigma)\leq 5$, contradicting the assumption of $F(G,\sigma)> \frac{3v(G)+2}{8}$. Since $d_{G^*}(v^c_{H_i})=2$ for each $i\in \{1,2,\ldots, m\}$, $(G^*, \sigma^*)$ contains at least $m$ vertices of degree $2$, each lying in a distinct triangle and all of which are pairwise vertex-disjoint. Then $v(G^*)\geq 3m$. Moreover, $(G^*, \sigma^*)$ is not isomorphic to any of $\widehat{\Gamma}_i$'s for $i\in \{1,2,\ldots, 5\}$. Therefore, as $(G^*, \sigma^*)$ has no leaf-block isomorphic to $\widehat \Gamma_2$, by~\Cref{prop:leaf-block}, $F(G^*, \sigma^*)\le \frac{1}{3}v(G^*)$. Thus we have $$\frac{F(G,\sigma)}{v(G)}=\frac{F(G^*, \sigma^*)+2m}{v(G^*)+5m}\le \frac{\frac{1}{3}v(G^*)+2m}{v(G^*)+5m}\leq \frac{3}{8},$$ where the last inequality follows from the fact of $v(G^*)\ge 3m$, a contradiction.
\end{proof}

The final part of our analysis is devoted to characterizing the family of signed graphs that achieve the bound $\frac{3v(G) + 2}{8}$.

\begin{proof}[Proof of the moreover part of \Cref*{thm:1ec}]

One direction is straightforward. If a signed connected simple graph $(G, \sigma)$ is cubic and contains at least one cut-edge, with each leaf-block isomorphic to $\widehat{\Gamma}_2$ and each non-leaf-block isomorphic to a negative triangle, then $(G, \sigma)$ can be constructed from a cubic tree $T$ (in which all non-leaf vertices have degree $3$) by replacing each $3$-vertex with a negative triangle, and attaching each leaf vertex to the unique $2$-vertex of a distinct copy of $\widehat{\Gamma}_2$. It is easy to see that if $|E^-_{(G, \sigma)}|=F(G, \sigma)$, then $(G, \sigma)$ has $8k+10$ vertices with exactly $3k+4$ negative edges, where $k$ is the number of $3$-vertices in $T$. Hence, $F(G,\sigma)=\frac{3v(G)+2}{8}$.

We now show the other direction.

\begin{claim}\label{claim:if}
If a signed connected simple subcubic graph $(G,\sigma)$ not isomorphic to $\widehat{\Gamma}_1$ satisfies that $F(G,\sigma)= \frac{3v(G)+2}{8}$, then $(G,\sigma)$ is cubic and contains at least one cut-edge, each leaf-block is isomorphic to $\widehat{\Gamma}_2$, and each other block is isomorphic to a negative triangle.
\end{claim}

\begin{proofofclaim}
Let $(G,\sigma)$ be a signed connected simple subcubic graph not isomorphic to $\widehat{\Gamma}_1$ such that $F(G,\sigma)= \frac{3v(G)+2}{8}$. 

Since $F(\widehat \Gamma_i,\sigma)< \frac{3v(G)+2}{8}$ for $i\in \{2,3,4,5\}$ and $(G, \sigma)$ is not isomorphic to $\widehat{\Gamma}_1$, by~\Cref{thm:main}, $(G, \sigma)$ contains at least one cut-edge.

Next we claim that every leaf-block of $(G,\sigma)$ is isomorphic to $\widehat \Gamma_2$. Otherwise, let $H$ be a leaf-block not isomorphic to $\widehat{\Gamma}_2$. Observing that $d_H(v_H)\leq 2$, $(H, \sigma|_{H})$ is not isomorphic to any of the exceptional signed graphs $\widehat{\Gamma}_i$'s. So by~\Cref{thm:main}, $F(H,\sigma|_{H})\le \frac{v(H)}{3}<\frac{3v(H)}{8}$. Then $F(G-H, \sigma|_{G-H})=F(G,\sigma)-F(H,\sigma|_{H})>\frac{3v(G)+2}{8}-\frac{3v(H)}{8} =\frac{3v(G-H)+2}{8}$. By the first part of~\Cref{thm:1ec}, $G-H$ must be isomorphic to $\widehat{\Gamma}_1$, which is not possible as $v^c_H$ is a $2$-vertex in $G-H$.

Moreover, we show that for each leaf-block $H$ of $(G, \sigma)$, $d_G(v_{H^c})=3$ and $v_{H^c}$ is contained in either a triangle or another leaf-block. Observe that $d_G(v_{H^c})\ge 2$. If $d_G(v_{H^c})=2$, let $v$ be the other neighbor of $v_{H^c}$. The new signed graph formed from $(G, \sigma)$ by deleting $v_{H^c}$ and adding a positive edge $v_H v$, is simple and clearly not isomorphic to $\widehat{\Gamma}_1$. It has the same frustration index as $(G, \sigma)$, but fewer vertices, contradicting the first part of~\Cref{thm:1ec}. If $v_{H^c}$ is in another leaf-block isomorphic to $\widehat{\Gamma}_2$, then $(G', \sigma')$ is isomorphic to $\widehat{\Gamma}_1$ and thus $(G, \sigma)$ is isomorphic to the signed graph in~\Cref{fig:leaf-leaf}. So we may assume to the contrary that $x$ and $y$ are two neighbors of $v_{H^c}$ in $G-H$ and $xy\not\in E(G)$. The new signed graph $(G', \sigma')$, formed from $(G, \sigma)$ by deleting $H$ and the vertex $v_{H^c}$, and adding an edge $xy$ with $\sigma'(xy)=\sigma(xv_{H^c})\sigma(yv_{H^c})$,  is simple, connected, and not isomorphic to $\widehat{\Gamma}_1$. We have that $F(G',\sigma')=F(G, \sigma)-F(H, \sigma|_H)>\frac{3v(G)+2}{8}-2=\frac{3(v(G')+6)+2}{8}-2>\frac{3v(G')+2}{8}$, a contradiction to the first part of~\Cref{thm:1ec}.

\medskip
Now we prove that each non-leaf-block of $(G, \sigma)$ is isomorphic to a negative triangle. Since the signed graph in~\Cref{fig:leaf-leaf} has no non-leaf-block, we can assume in the following that for each leaf-block $H$ of $(G, \sigma)$, $v_{H^c}$ is contained in a triangle. 
Assume to the contrary that $(G, \sigma)$ contains a non-leaf-block $Q$ that is not isomorphic to a negative triangle. Among all such signed graphs, we select one with the smallest number of vertices. Let $H$ be a leaf-block of $G$ and assume that $v_{H^c}$ is in a triangle $v_{H^c}xy$. 

We next show that each of $x$ and $y$ is of degree $3$. Assume not and at least one of them is a $2$-vertex, say $x$. Note that $y$ cannot be of degree $2$, as otherwise, $v_{H^c}xy$ is a leaf-block not isomorphic to $\widehat \Gamma_2$, a contradiction. Let $y'$ denote the third neighbor of $y$. We form a new signed graph $(G',\sigma')$ from $(G, \sigma)$ by deleting $H,v_{H^c},x,$ and $y$. Since $yy'$ is a cut-edge and $F(G[\{v_{H^c},x,y\}, \sigma])\leq 1$, we have that $F(G',\sigma')\geq F(G, \sigma)-3=\frac{3v(G)+2}{8}-3=\frac{3(v(G')+8)+2}{8}-3=\frac{3v(G')+2}{8}.$ Moreover, since $d_{G'}(y')\leq 2$, $(G',\sigma')$ is not isomorphic to $\widehat{\Gamma}_1$. Thus by the first part of~\Cref{thm:1ec}, $F(G',\sigma')=\frac{3v(G')+2}{8}$. It implies that $F(G[\{v_{H^c},x,y\}, \sigma])=1$, i.e., $v_{H^c}xy$ is a negative triangle. Moreover, since $yy'$ is a cut-edge, $G[V(Q)]=G'[V(Q)]$. If $G'[V(Q)]$ is a leaf-block, then $y'\in V(Q)$. As $d_{G'}(y')\leq 2$, $G'[V(Q)]$ is not isomorphic to $\widehat{\Gamma}_2$. Noting that $(G', \sigma')$ satisfies that  $F(G',\sigma')=\frac{3v(G')+2}{8}$, it is a contradiction to the previous claim. So $Q$ is still a non-leaf-block of $(G', \sigma')$ and thus $(G', \sigma')$ is a smaller counterexample, a contradiction.

Let $x'$ and $y'$ be the third neighbors of $x$ and $y$, respectively. We claim that $x' \neq y'$. Suppose to the contrary that $x' = y'$. Observe that $F(G[\{v_{H^c}, x, y, x'\}], \sigma) \leq 1$. If $x'$ is a $2$-vertex, then $F(G, \sigma)<\frac{3v(G) + 2}{8}$, a contradiction. Hence, we may assume that $x'$ is a $3$-vertex. Let $x''$ denote the other neighbor of $x'$. Let $(G'', \sigma'')$ be the signed graph obtained from $(G, \sigma)$ by deleting the leaf-block $H$ and the vertices $v_{H^c}$, $x$, $y$, and $x'$. Since $x'x''$ is a cut-edge, we have $F(G'', \sigma'') \geq F(G, \sigma) - 3$. Moreover, $(G'', \sigma'')$ is not isomorphic to $\widehat{\Gamma}_1$, for otherwise $(G, \sigma)$ would have $14$ vertices satisfying that $F(G, \sigma) \leq 5$, contradicting the assumption that $F(G, \sigma) = \frac{3v(G) + 2}{8}$.
Now, $F(G'', \sigma'') \geq F(G, \sigma) - 3 = \frac{3v(G) + 2}{8} - 3 = \frac{3(v(G'') + 9) + 2}{8} - 3 > \frac{3v(G'') + 2}{8},$
which contradicts the first part of~\Cref{thm:1ec}. 

\smallskip
We consider the following two cases based on whether $x'y'$ is an edge or not. From now on, we choose $\sigma$ to be the signature such that $|E^-_{(G,\sigma)}|=F(G,\sigma)$.

\smallskip
\noindent
\textbf{Case 1} $x'y'\in E(G)$. In this case, at least one of $x'$ and $y'$ is of degree $3$, as otherwise, $G[\{v_H, x,y,x',y'\}]$ is a leaf-block that is not isomorphic to $\widehat{\Gamma}_2$, a contradiction. Without loss of generality, we assume that $d_G(y')=3$. Let $x''$ be the third neighbor of $x'$ if it exists. We form $(G_1,\sigma_1)$ from $(G, \sigma)$ by deleting $H,v_{H^c},x,y$, and $x'$, and adding an edge $x''y'$ with $\sigma_1(x''y')=\sigma(x'x'')\sigma(x'y')$ if $x''$ exists. Since $d_{G_1}(y')\leq 2$, $(G_1,\sigma_1)$ cannot be isomorphic to $\widehat{\Gamma}_1$. Noting that in $(G,\sigma)$ there is at most one negative edge among the edges $v_{H^c}x,v_{H^c}y,xy,xx'$, and $yy'$, we have $$F(G_1,\sigma_1)\ge F(G, \sigma)-3=\frac{3v(G)+2}{8}-3=\frac{3(v(G_1)+9)+2}{8}-3>\frac{3v(G_1)+2}{8},$$ a contradiction to the first part of~\Cref{thm:1ec}. 

\smallskip
\noindent
\textbf{Case 2} $x'y'\notin E(G)$. We form $(G_2,\sigma_2)$ from $(G, \sigma)$ by deleting $H,v_{H^c},x$, and $y$, and adding an edge $x'y'$ with $\sigma_2(x'y')=\sigma(xx')\sigma(yy')$. Noting that $F(G[\{v_{H^c},x,y\}, \sigma])\leq 1$, we have $$F(G_2,\sigma_2)\ge F(G, \sigma)-3=\frac{3v(G)+2}{8}-3=\frac{3(v(G_2)+8)+2}{8}-3=\frac{3v(G_2)+2}{8}.$$ Now $(G_2,\sigma_2)$ is not isomorphic to $\widehat{\Gamma}_1$, as otherwise, $(G, \sigma)$ contains $12$ vertices satisfying that $F(G, \sigma)\leq 4$, a contradiction to the assumption of $F(G, \sigma)=\frac{3v(G)+2}{8}$. Hence, by the first part of~\Cref{thm:1ec}, $F(G_2,\sigma_2)=\frac{3v(G_2)+2}{8}$. So $v_{H^c}xy$ is a negative triangle. 

If $(G_2, \sigma_2)$ contains a non-leaf-block not isomorphic to a negative triangle, then $(G_2, \sigma_2)$ is a smaller counterexample, contradicting the minimality. So we assume that every non-leaf-blocks of $(G_2,\sigma_2)$ is isomorphic to a negative triangle. Since $Q\subset G$, $\{x',y'\}\subset V(Q)$ and thus $x'$ and $y'$ are in the same block, say $Q'$, of $(G_2, \sigma_2)$. By the minimality of $(G, \sigma)$, either $Q'$ is a leaf-block of $(G_2, \sigma_2)$, or $Q'$ is a negative triangle. 
\begin{itemize}
\setlength{\itemsep}{0.1em}
\item In the former case, since $F(G_2, \sigma_2)=\frac{3v(G_2)+2}{8}$, if $Q'$ is a leaf-block of $(G_2, \sigma_2)$, then $Q'=\widehat{\Gamma}_2$. Thus  $e_{Q'}$ is a cut-edge in $(G,\sigma)$, and neither $x'$ nor $y'$ is an endpoint of $e_{Q'}$. We denote two connected components of $G-e_{Q'}$ by $H_1$ and $H_2$, where $v_H\in H_1$. As $v_{Q'^c}$ is a $2$-vertex in $H_2$, $(H_2, \sigma|_{H_2})$ is not isomorphic to $\widehat \Gamma_1$. Noting that $H_1$ has $13$ vertices with $F(H_1,\sigma|_{H_1})\leq 4$, we have that {\small $$ F(H_2,\sigma|_{H_2})= F(G, \sigma)-F(H_1,\sigma|_{H_1})\ge\frac{3v(G)+2}{8}-4=\frac{3(v(H_2)+13)+2}{8}-4>\frac{3v(H_2)+2}{8},$$} a contradiction to the first part of~\Cref{thm:1ec}. 

\item In the latter case, if $Q'$ is a negative triangle, then the block $Q$ consists of $Q'$ and a 5-cycle that share the common edge $xy$. Note that $F(Q,\sigma)\leq 1$. Since $v_{H^c}xy$ forms a negative triangle, the 5-cycle cannot contain any negative edge other than $xy$; in particular, the edges $xx'$ and $yy'$ are both positive. Now, since $Q'$ is a negative triangle, the only edge in $(G_2, \sigma_2)$ that can be negative is $x'y'$. However, recalling that $\sigma_2(x'y')=\sigma(xx')\sigma(yy')$, it follows that one of $xx'$ or $yy'$ must be negative in $(G,\sigma)$, which leads to a contradiction.
\end{itemize}
Hence, we conclude that each non-leaf-block of $(G, \sigma)$ is isomorphic to a negative triangle.

\smallskip    
Finally, we show that $(G, \sigma)$ is cubic. Assume to the contrary that $x$ is a $2$-vertex of $G$. The above discussion implies that $x$ is in a non-leaf-block isomorphic to a negative triangle, denoted $xyz$. Note that if one of $y$ and $z$ is also $2$-vertex, then it is a leaf-block not isomorphic to $\widehat \Gamma_2$, a contradiction. So we let $y'$ and $z'$ be the third neighbor of $y$ and $z$, respectively. As $xyz$ is a block, $y'z'\not\in E(G)$. We form a new signed graph $(G''',\sigma''')$ from $(G, \sigma)$ by deleting $x,y$ and $z$, and adding a positive edge $y'z'$. We have that $$F(G''',\sigma''')= F(G, \sigma)-1=\frac{3v(G)+2}{8}-1=\frac{3(v(G''')+3)}{8}-1>\frac{3v(G''')+2}{8}.$$ Note that $(G''',\sigma''')$ is simple and connected. Moreover, $(G''',\sigma''')$ is not isomorphic to $\widehat \Gamma_1$, as otherwise, $(G,\sigma)$ contains $7$ vertices satisfying that $F(G, \sigma)=3$, a contradiction to the assumption that $F(G,\sigma)=\frac{3v(G)+2}{8}$. Now $(G''',\sigma''')$ is a smaller counterexample, contradicting the first part of~\Cref{thm:1ec}.
\end{proofofclaim}
We complete the proof of~\Cref{thm:1ec}.
\end{proof}

\section[Small]{Signed $2$-edge-connected simple subcubic graphs on at most $9$ vertices}\label{sec:9vertices}

In this section, we characterize $2$-edge-connected simple signed subcubic graphs $(G, \sigma)$ on at most $9$ vertices. 
Recall that it follows from~\Cref{lem:bad-cut} that for every signed subcubic graph $(G, \sigma)$, $F(G, \sigma) \leq \frac{1}{2}v(G)$. In the following, we first characterize all signed 2-edge-connected subcubic graphs for which this bound is attained; we emphasize that parallel edges are allowed in this result.

\begin{lemma}\label{lem:parallel}
Let $(G, \sigma)$ be a signed $2$-edge-connected subcubic graph such that $|E^-_{(G, \sigma)}|=F(G, \sigma)$. If $F(G,\sigma)=\frac{1}{2}v(G)$, then either $(G, \sigma)$ is isomorphic to $\widehat{\Gamma}_1$ up to switching, or each vertex of $(G, \sigma)$ is in a digon.
\end{lemma}

\begin{proof}
Let $(G, \sigma)$ be a signed $2$-edge-connected subcubic graph such that $F(G,\sigma)=\frac{1}{2}v(G)$. Moreover, assume that $|E^-_{(G, \sigma)}|=F(G,\sigma)$. Now $E^-_{(G, \sigma)}=\{u_1v_1,\cdots,u_kv_k\}$ is a perfect matching of $G$. 
Note that every vertex $x$ is adjacent to exactly one negative edge of $E^-_{(G, \sigma)}$, and we denote the negative edge by $e_x$. Moreover, we use $x'$ and $x''$ to denote the other two positive neighbors of $x$. We first claim that $e_x,e_{x'}$, and $e_{x''}$ cannot be distinct. Otherwise, 
$G[\{x,x',x''\}]\in \mathcal{T}\cup \mathcal{C}$ is an induced subgraph of $G$ (noting that in the latter case $G[\{x,x',x''\}]$ is a positive triangle), a contradiction to~\Cref{lem:tree}. If $x'=x''$, then $G[\{x,x'\}]\in \mathcal{C}$ is an induced subgraph of $G$, again a contradiction to~\Cref{lem:tree}. Hence, either $x$ is in a digon (i.e., either $e_x=e_{x'}$ or $e_x=e_{x''}$), or there is a negative triangle in $G[\{x,x',x''\}]$ (i.e., $e_{x'}=e_{x''}$). 

We shall prove that if the latter case occurs ($e_{x'}=e_{x''}$), then $(G,\sigma)$ is isomorphic to $\widehat{\Gamma}_1$ up to switching. We denote the other positive neighbor of $x'$ and $x''$ by $y'$ and $y''$ respectively, and denote the negative neighbor of $x$ by $y$. As both $x$ and $y'$ are positive neighbors of $x'$, the negative edges $e_{x},e_{x'}$, and $e_{y'}$ are not distinct, and thus $e_{y'}\in \{e_{x},e_{x'}\}$, that is to say, $y'\in \{x,x'',y\}$. Note that $y'\ne x$, as otherwise $x$ is adjacent to three positive edges, a contradiction. If $y'=x''$, then between $x'$ and $x''$ there is a digon, and thus $xy$ is a cut-edge, a contradiction. So $y'=y$. Now as both $x$ and $y''$ are positive neighbors of $x''$, the negative edges $e_{x},e_{x''}$, and $e_{y''}$ are not distinct. Similarly as $y''\notin \{x',x\}$, $y''=y$. Now $\{x,x',x'',y\}$ forms a $K_4$, and as $xy$ and $x'x''$ are both negative, $(G, \sigma)$ is isomorphic to $\widehat{\Gamma}_1$.
\end{proof}

Note that for each $\widehat{\Gamma}_i$ in~\Cref{fig:Exception}, its signature is the one that achieves its frustration index. That is because all of them are $2$-edge-connected planar graphs, so each edge is contained in exactly two facial cycles. Thus the number of negative facial cycles should be at most twice of the frustration index. Moreover, by~\Cref{lem:parallel}, for signed simple subcubic graphs $(G, \sigma)$, either $F(G,\sigma)<\frac{1}{2}v(G)$ or $(G,\sigma)$ is isomorphic to $\widehat{\Gamma}_1$, so any signature $\pi$ on $\Gamma_i$ would imply that $F(\Gamma_i, \pi)\leq \frac{1}{3}v(\Gamma_i)$ for $i\in \{2,3,4,5\}$.

\medskip
Now we claim that all the other signed $2$-edge-connected simple subcubic graphs $(G, \sigma)$ on at most $9$ vertices satisfy $F(G, \sigma)\leq \frac{1}{3}v(G)$. We begin by considering cubic graphs, noting that such graphs must have an even number of vertices.

For $v(G)= 4$, the complete graph $K_4$ is the unique cubic graph on $4$ vertices. Here, $\frac{1}{3} v(K_4)<F(K_4, \sigma)\leq \frac{1}{2}v(K_4)$ holds if and only if the signature $\sigma$ satisfies that $|E^-_{(G, \sigma)}|=2$ and thus $E^-_{(G, \sigma)}$ forms a perfect matching, as shown in~\Cref{fig:K4}. 

For $v(G)=6$, by~\Cref{lem:parallel} we have $F(G,\sigma)<\frac{1}{2}v(G)=3$, i.e., $F(G,\sigma)\leq \frac{1}{3}v(G)=2$.

For $v(G)= 8$, it is proved in \cite{BCCS1977} that there are only five such cubic graphs. We have demonstrated that three of them (see~\Cref{fig:Gamma1,fig:Gamma2,fig:Cube}) admit signatures that serve as exceptional signed graphs. It remains to show that the remaining two graphs, $W_1$ and $W_2$, depicted in~\Cref{fig:v=8}, satisfy $F(W_i, \sigma) \leq \frac{1}{3} v(W_i)$ for any arbitrary signature $\sigma$.

\begin{lemma}
Let $W_1$ and $W_2$ be two graphs as depicted in~\Cref{fig:W1} and \Cref{fig:W2}. For each $i\in \{1,2\}$ and any signature $\sigma$, $F(W_i, \sigma) \leq \frac{1}{3} v(W_i)$.
\end{lemma}

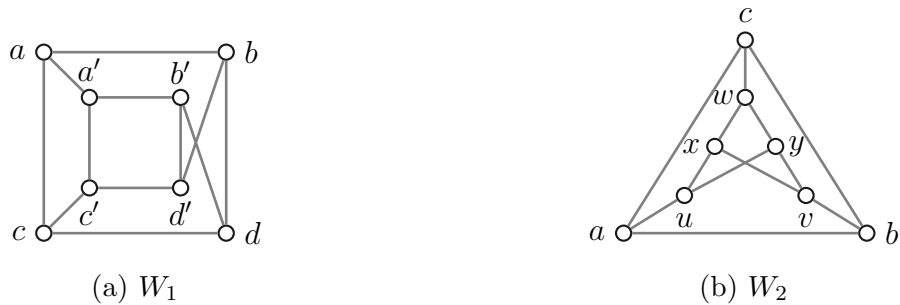
\begin{figure}[htbp]
\begin{subfigure}[t]{.48\textwidth}
		\centering
	     \begin{tikzpicture}[>=latex,
		roundnode/.style={circle, draw=black!90, thick, minimum size=2mm, inner sep=0pt},
        squarenode/.style={rectangle, draw=black!90, thick, minimum size=2mm, inner sep=0pt},
        scale=0.8
		]
          \node[roundnode] (a) at (0,0) {};
          \node[roundnode] (b) at (3,0) {};
          \node[roundnode] (c) at (0,-3) {};
          \node[roundnode] (d) at (3,-3) {};
          \node[roundnode] (a1) at (0.75,-0.75) {};
          \node[roundnode] (b1) at (2.25,-0.75) {};
          \node[roundnode] (c1) at (0.75,-2.25) {};
          \node[roundnode] (d1) at (2.25,-2.25) {};
        
          \draw[fill=white,line width=0.2pt] (0,0) node[left=1mm] {$a$};
          \draw[fill=white,line width=0.2pt] (3,0) node[right=1mm] {$b$};
          \draw[fill=white,line width=0.2pt] (0,-3) node[left=1mm] {$c$};
          \draw[fill=white,line width=0.2pt] (3,-3) node[right=1mm] {$d$};
          \draw[fill=white,line width=0.2pt] (0.75,-0.75) node[above=0.7mm] {$a'$};
          \draw[fill=white,line width=0.2pt] (2.25,-0.75) node[above=1pt] {$b'$};
          \draw[fill=white,line width=0.2pt] (0.75,-2.25) node[below=1pt] {$c'$};
          \draw[fill=white,line width=0.2pt] (2.25,-2.25) node[below=1pt] {$d'$};

          \draw[line width=1pt, gray] (a)--(b)--(d)--(c)--(a)--(a1)--(c1)--(c);
          \draw[line width=1pt, gray] (a1)--(b1)--(d1)--(c1);
          \draw[line width=1pt, gray] (b)--(d1);
          \draw[line width=1pt, gray] (d)--(b1);

	    \end{tikzpicture} 
        \caption{$W_1$}
		\label{fig:W1}
  \end{subfigure}
  \begin{subfigure}[t]{.48\textwidth}
  \centering
		\begin{tikzpicture}[>=latex,
		roundnode/.style={circle, draw=black!90, thick, minimum size=2mm, inner sep=0pt},
        squarenode/.style={rectangle, draw=black!90, thick, minimum size=2mm, inner sep=0pt},
        scale=0.8
		]
          \node[roundnode] (a) at (0,0) {};
          \node[roundnode] (b) at (4,0) {};
          \node[roundnode] (c) at (2,3.2) {};
          \node[roundnode] (u) at (1,0.625) {};
          \node[roundnode] (v) at (3,0.625) {};
          \node[roundnode] (w) at (2,2.25) {};
          \node[roundnode] (x) at (1.5,1.4375) {};
          \node[roundnode] (y) at (2.5,1.4375) {};

          \draw[fill=white,line width=0.2pt] (0,0) node[left=1mm] {$a$};
          \draw[fill=white,line width=0.2pt] (4,0) node[right=1mm] {$b$};
          \draw[fill=white,line width=0.2pt] (2,3.2) node[above=1mm] {$c$};
          \draw[fill=white,line width=0.2pt] (1,0.625) node[below=1mm] {$u$};
          \draw[fill=white,line width=0.2pt] (3,0.625) node[below=1mm] {$v$};
          \draw[fill=white,line width=0.2pt] (2,2.25) node[left=0mm] {$w$};
          \draw[fill=white,line width=0.2pt] (1.5,1.4375) node[left=0.5mm] {$x$};
          \draw[fill=white,line width=0.2pt] (2.5,1.4375) node[right=0.3mm] {$y$};

          \draw[line width=1pt, gray] (a)--(b)--(c)--(a)--(u)--(x)--(w)--(y)--(v)--(b);
          \draw[line width=1pt, gray] (u)--(y);
          \draw[line width=1pt, gray] (v)--(x);
          \draw[line width=1pt, gray] (c)--(w);
	    \end{tikzpicture} 
        \caption{$W_2$}
		\label{fig:W2}
    \end{subfigure}
    \caption{Two $8$-vertex cubic graphs $W_1$ and $W_2$}  
    \label{fig:v=8}
\end{figure}

\begin{proof}
Assume to the contrary that for each $i\in \{1,2\}$, there is a signature $\sigma_i$ such that $F(W_i, \sigma_i)> \frac{1}{3}v(W_i)$. Moreover, we may assume that $|E^-_{(W_i, \sigma_i)}|=F(W_i, \sigma_i)$. Since $W_1$ and $W_2$ has no parallel edges, and neither of them is isomorphic to $K_4$, \Cref{lem:parallel} implies that $F(W_i, \sigma_i)< \frac{1}{2}v(W_i)$, and thus $3\leq |E^-_{(W_i, \sigma_i)}|< 4$, i.e., $|E^-_{(W_i, \sigma_i)}|=3$. We consider two signed graphs case by case:

\smallskip
\noindent
{\bf Case 1}. For $(W_1, \sigma_1)$, we first claim that the cycle $aa'c'c$ cannot be all-positive under $\sigma_1$.

Suppose for contradiction that the cycle $aa'c'c$ is all-positive in $(W_1, \sigma_1)$. We consider two possibilities based on the sign of $bd$. 

Assume that $bd$ is negative under $\sigma_1$. In this case, by~\Cref{lem:bad-cut}, $ab,bd',b'd$, and $cd$ are all positive edges. We claim that $a'b'$ is positive. Assume not, and both $b'd$ and $b'd'$ are positive. As $|E^-_{(W_1, \sigma_1)}|=3$, $c'd'$ must be negative. However, it implies that $W_1[\{b',d',b\}]\in \mathcal{T}$ is an induced subgraph of $W_1$, a contradiction to~\Cref{lem:tree}. Since $a'b'$ is positive and $|E^-_{(W_1, \sigma_1)}|=3$, both $c'd'$ and $d'b'$ are negative, a contradiction to~\Cref{lem:bad-cut}. 

Assume that $bd$ is positive under $\sigma_1$. By symmetry, $b'd'$ is also a positive edge. Now consider the remaining six edges, by~\Cref{lem:bad-cut} there are only two possibilities under symmetry: either $E^-_{(W_1, \sigma_1)}=\{ab,b'd,c'd'\}$, or $E^-_{(W_1, \sigma_1)}=\{ab,b'd',cd\}$. But in the previous case, $W_1[\{b',d',b\}]\in \mathcal{T}$ is an induced subgraph of $W_1$ and in the latter case, $W_1[\{d',b',d\}]\in \mathcal{T}$ is an induced subgraph of $W_1$, both contradicting~\Cref{lem:tree}. 

Hence, we conclude that the cycle $aa'c'c$ cannot be all-positive under $\sigma_1$. By symmetry, we may assume that either $ac$ or $aa'$ is negative in $(W_1, \sigma_1)$. 
\begin{itemize}
\itemsep 0em
    \item  If $ac$ is negative, then by~\Cref{lem:bad-cut}, edges $ab,aa',cc',$ and $cd$ are all positive. We first claim that $a'c'$ is also positive. Assume not, as $|E^-_{(W_1, \sigma_1)}|=3$, there exists a negative edge $e$ that is adjacent to neither $ac$ nor $a'c'$, thus one of $\{b,b'\}$ is in $e$. Now either $W_1[\{a,a',b'\}]$ or $W_1[\{a',a,b\}]\in \mathcal{T}$ is an induced subgraph of $W_1$, a contradiction to~\Cref{lem:tree}. We then claim that $a'b'$ is also positive. Assume not, as $|E^-_{(W_1, \sigma_1)}|=3$, there exists a negative edge $e$ that is adjacent to neither $ac$ nor $a'b'$, thus one of $\{b,c'\}$ is in $e$. Now either $W_1[\{a,a',c'\}]\in \mathcal{T}$ or $W_1[\{b,a,a'\}]\in \mathcal{T}$ is an induced subgraph of $W_1$, a contradiction to~\Cref{lem:tree}. So $c'd'$ is positive by symmetry. Now we consider the remaining edges, which are precisely those in the $4$-cycle $bd'b'd$. As $|E^-_{(W_1, \sigma_1)}|=3$, either $\{bd,b'd'\}\subset E^-_{(W_1, \sigma_1)}$ or $\{bd',b'd\}\subset E^-_{(W_1, \sigma_1)}$, and thus either $W_1[\{a,b,d'\}]\in \mathcal{T}$ or $W_1[\{a,b,d\}]\in \mathcal{T}$ is an induced subgraph of $G$, a contradiction~\Cref{lem:tree}.
    
    \item If $aa'$ is negative, then by~\Cref{lem:bad-cut}, edges $ac,ab,a'c',$ and $a'b'$ are all positive. We claim that the edge-cut $[a,b]$ cannot contain two negative edges. Assume not, we may switch at $[a,b]$ and denote the resulting signature by $\sigma_1'$. Note that $|E^-_{(W_1, \sigma_1)}|=|E^-_{(W_1, \sigma_1')}|$. Now $ac$ is negative in $(W_1,\sigma_1')$, which has been discussed above. By symmetry, the edge-cut $[a',b']$ cannot contain two negative edges. Thus edges $bd,b'd',b'd,$ and $bd'$ are all positive. Among the remaining edges $\{cd,cc',c'd'\}$, as $|E^-_{(W_1, \sigma_1)}|=3$, $cd$ and $c'd'$ are negative, and thus $W_1[\{a',c',c\}]\in \mathcal{T}$ is an induced subgraph of $W_1$, a contradiction to~\Cref{lem:tree}.
\end{itemize}

\noindent
{\bf Case 2}. For $(W_2, \sigma_2)$, We first claim that the triangle $abc$ of $W_2$ cannot be all-positive.

Assume to the contrary that $abc$ is all-positive in $(W_2, \sigma_2)$. Since $\{au,bv,cw\}$ is an edge-cut, at most one of them can be negative under $\sigma_2$. If $au$ is negative, then by~\Cref{lem:bad-cut} $ux$ and $uy$ are positive. We consider edges in the cycle $xwyv$, as the other edges are positive except $ua$. As $|E^-_{(W_2, \sigma_2)}|=3$, either $\{xw,yv\}\subset E^-_{(W_2, \sigma_2)}$ or $\{xv,wy\}\subset E^-_{(W_2, \sigma_2)}$, and thus either $W_2[\{u,x,w\}]\in \mathcal{T}$ or $W_2[\{u,y,w\}]\in \mathcal{T}$ is an induced subgraph of $W_2$, a contradiction to~\Cref{lem:tree}. So $au$ is positive and by symmetry, $bv$ and $cw$ are both positive. But now the subgraph induced on $\{u,v,x,y,w\}$ is simple and has five vertices, and by~\Cref{lem:bad-cut} it contains at most $2$ negative edges. So it contradicts the fact that $|E^-_{(W_2, \sigma_2)}|=3$.

Hence, the triangle $abc$ of $W_2$ cannot be all-positive. Note that $x$ and $y$ are both adjacent to $u,v$, and $w$. By symmetry, we may assume that $ab$ is negative in $(W_2, \sigma_2)$. Now by~\Cref{lem:bad-cut}, edges $au,ac,bv,$ and $bc$ are all positive. We claim that $cw$ is also positive. Assume not, then by~\Cref{lem:bad-cut} edges $xw$ and $yw$ are all positive. As $|E^-_{(W_2, \sigma_2)}|=3$, there exists a negative edge $e$ such that either $u$ or $v$ is an endpoint of $e$, thus either $W_2[\{u,a,c\}]\in \mathcal{T}$ or $W_2[\{v,b,c\}]\in \mathcal{T}$ is an induced subgraph of $W_2$, a contradiction to~\Cref{lem:tree}. We then claim that $ux$ is positive. Assume not, by~\Cref{lem:bad-cut} $uy,xw,$ and $xv$ are all positive. As $|E^-_{(W_2, \sigma_2)}|=3$, either $wy$ or $vy$ is negative, and thus either $W_2[\{a,u,y\}]\in \mathcal{T}$ or $W_2[\{b,v,x\}]\in \mathcal{T}$ is an induced subgraph of $W_2$, a contradiction to~\Cref{lem:tree}. Now by symmetry, $uy,vx,$ and $vy$ are all positive. Since $|E^-_{(W_2, \sigma_2)}|=3$, $xw$ and $yw$ must be both negative, a contradiction to~\Cref{lem:bad-cut}.
\end{proof}

Next, we consider signed $2$-edge-connected simple  subcubic graphs containing $2$-vertices, where the total number of vertices may be odd (or even).

\begin{observation}\label{lem:subdivision}
Let $(G, \sigma)$ be a signed graph and let $(G', \sigma')$ be formed from $(G, \sigma)$ by subdividing some edges. Then $F(G',\sigma')= F(G,\sigma)$.
\end{observation}

Observe that any signed tree has frustration index $0$. So we only consider signed subcubic graphs with at least one cycle.

If $v(G)=3$, then $G$ is isomorphic to $3$-cycle. It is easy to see that there are only two signed triangles $(C_3, +)$ and $(C_3, -)$ up to switching, where $F(C_3, +)=0$ and $F(C_3, -)=1$. Thus for any signature $\sigma$, $F(C_3, \sigma)\leq \frac{1}{3}v(C_3)$.

If $v(G)\in \{5, 7, 9\}$, then $G$ is a subcubic graph, and can be obtained from subdividing $k$ edges of a cubic graph $G'$, where $k$ is an odd number. For any signature $\sigma$ on $G$, let $\pi$ denote the corresponding signature on $G'$ that is induced by $\sigma$ before the subdivision is performed. If $k\ge 3$, then $v(G')\le v(G)-3\le 6$. As $F(G', \pi)\le \frac{1}{2}v(G')$ for any signature $\pi$ and by~\Cref{lem:subdivision}, we have $F(G, \sigma)= F(G',\pi)\le \frac{1}{2}v(G') \le \frac{1}{3}v(G)$. If $k=1$, then $(G', \pi)$ has at most one digon. As $v(G')\geq 4$, there exist some vertices that are not in any digon. So by~\Cref{lem:parallel}, either $(G', \pi)$ is isomorphic to $\widehat \Gamma_1$ or $F(G',\pi)<\frac{1}{2}v(G')$. Assume that there is a signature $\sigma$ such that $F(G,\sigma)>\frac{1}{3}v(G)$. By~\Cref{lem:subdivision}, if $v(G)=9$, then $F(G,\sigma)= F(G',\pi)<\frac{1}{2}v(G')=4$, a contradiction; if $v(G)=7$, then $F(G,\sigma)= F(G',\pi)<\frac{1}{2}v(G')=3$, a contradiction. So $v(G)=5$, and $F(G,\sigma)=F(G',\pi)=\frac{1}{2}v(G')=2$. This implies that $G$ is a subdivision of $\widehat{\Gamma}_1$, as shown in~\Cref{fig:W}.

For $v(G)\in \{4,6,8\}$, it can be formed from $G'$ where $v(G')\in \{3,5,7\}$ by subdividing some edges. By~\Cref{lem:subdivision}, $F(G,\sigma)= F(G',\pi)\le \frac{1}{3}v(G')< \frac{1}{3}v(G)$.

\section[KeyLemma]{Proof of~\Cref*{lem:key}}\label{sec:keylemma}

In this section, we shall first provide some reducible configurations for the minimum counterexample $(G, \sigma)$ to~\Cref{thm:main} and then prove~\Cref{lem:key}. Based on the discussion in~\Cref{sec:9vertices}, in the sequel, we may assume that $v(G)\geq 10$.

We first give some lemmas related to the exceptional signed graphs $\widehat{\Gamma}_i$. A signed graph $(H, \pi)$ is said to be \emph{critically $k$-frustrated} if $F(H, \pi)=k$ and, for any edge $e\in E(H)$, $F(H-e,\sigma|_{H-e})<k$. It follows from Lemma 4.3 of \cite{CS2022} that $\widehat{\Gamma}_1$ and $\widehat{\Gamma}_2$ are critically $2$-frustrated, and from Lemma 3.5 of \cite{CNSW2025} that $\widehat{\Gamma}_3, \widehat{\Gamma}_4$ and $\widehat{\Gamma}_5$ are critically $3$-frustrated. 

\begin{lemma}\label{lem:critical-equilibrated}{\rm \cite{CS2022}}
Let $k$ be a positive integer. A signed graph $(H, \pi)$ is critically $k$-frustrated if and only if every positive edge of $(H, \pi)$ is contained in an equilibrated cut.
\end{lemma}

By~\Cref{lem:critical-equilibrated}, we immediately have the following result.

\begin{corollary}\label{col:equilibrated}
For $i\in \{1,2,3,4,5\}$, every positive edge of $\widehat{\Gamma}_i$ is contained in an equilibrated cut.
\end{corollary}

\subsection[Property]{Properties of the minimum counterexample $(G, \sigma)$ to \Cref*{thm:main}} 
\label{sec:properties}

\begin{lemma}\label{lem:cubic}
The minimum counterexample $(G,\sigma)$ is cubic.
\end{lemma}

\begin{proof}
Since $(G, \sigma)$ is $2$-edge-connected, there is no $1$-vertex in $(G, \sigma)$.
Assume that $z$ is a $2$-vertex with two neighbors $x$ and $y$. We first claim that $xy\in E(G)$. Suppose for contradiction that $xy\notin E(G)$, we then form a signed graph $(G', \sigma')$ from $(G, \sigma)$ by deleting the vertex $z$ and adding an edge $xy$ such that $\sigma'(xy)= \sigma(xz)\sigma(zy)$. Note that $(G', \sigma')$ is $2$-edge-connected and simple (i.e., no parallel edges is created). Moreover, $|E^-_{(G', \sigma')}|=F(G', \sigma')$ and thus $$F(G', \sigma')= F(G, \sigma) > \frac{1}{3} v(G)= \frac{1}{3}(v(G')+1)>\frac{1}{3}v(G').$$ Since $v(G')=v(G)-1\geq 9$, $(G', \sigma')$ cannot be isomorphic to any of the five exceptional signed graphs $\widehat{\Gamma}_i$'s, and thus $(G', \sigma')$ is a smaller counterexample, contradicting the minimality of $(G, \sigma)$. So $xy\in E(G)$. Furthermore, since $(G, \sigma)$ is $2$-edge-connected, each of $x$ and $y$ is a $3$-vertex. Let $x'$ and $y'$ denote the third neighbor of $x$ and $y$, respectively. We then consider the following two cases.

\smallskip
\noindent
{\bf Case 1.} $xy\in E(G)$ and the triangle $xyz$ is positive. 
\smallskip

In this case, we form a signed graph $(G'',\sigma'')$ from $(G,\sigma)$ by deleting vertices $y$ and $z$, and adding an edge $xy'$ with  $\sigma''(xy')=\sigma(yy')$. Now no parallel edge is created and $G''$ is $2$-edge-connected. 

Since for any unequilibrated cut of $(G'', \sigma'')$ containing the edge $xy'$, replacing $xy'$ with $yy'$ results in an unequilibrated cut in $(G, \sigma)$, it follows that $(G'', \sigma'')$ contains no unequilibrated cut, and hence $F(G'',\sigma'')=|E^-_{(G'',\sigma'')}|$. Now we have that $$F(G'', \sigma'')= F(G, \sigma) > \frac{1}{3} v(G)= \frac{1}{3}(v(G'')+2)>\frac{1}{3}v(G'').$$ Note that $v(G'')=v(G)-2\geq 8$. If $(G'', \sigma'')$ is isomorphic to one of the five exceptional signed graphs, then 
$(G'', \sigma'')$ can only be one of $\widehat{\Gamma}_3, \widehat{\Gamma}_4$, and $\widehat{\Gamma}_5$, as shown in~\Cref{fig:Gamma1,fig:Gamma2,fig:Cube}. Observe that none of $\widehat{\Gamma}_3, \widehat{\Gamma}_4$, and $\widehat{\Gamma}_5$ has any $2$-vertex. However, $d_{G''}(x)=2$, which leads to a contradiction. Thus $(G'', \sigma'')$ is a smaller counterexample, contradicting the minimality of $(G, \sigma)$.

\smallskip
\noindent
{\bf Case 2.} $xy\in E(G)$ and the triangle $xyz$ is negative. 
\smallskip

In this case, without loss of generality, exactly one of $xy$ and $xz$ is negative. By~\Cref{lem:bad-cut}, both $xx'$ and $yy'$ are positive. We consider the following two subcases: If $x'y'\notin E(G)$, then let $(G_1,\sigma_1)$ be formed from $(G, \sigma)$ by deleting vertices $x,y$, and $z$, and adding a positive edge $x'y'$; If $x'y'\in E(G)$, then let $(G_2,\sigma_2)$ be formed from $(G, \sigma)$ by deleting vertices $x,y$, and $z$. Note that no parallel edge is created, and in each case the resulting signed graph $(G_i, \sigma_i)$ remains $2$-edge-connected. 

We next verify that $F(G_i, \sigma_i)=|E^-_{(G_i, \sigma_i)}|$ for $i\in \{1,2\}$. Assume not and it implies that $\sigma_i$ is not the signature that achieves the frustration index of $(G_i,\sigma_i)$. So $(G_i,\sigma_i)$ contains an unequilibrated cut $C_i$. If $C_i$ does not contain the edge $x'y'$, then itself is an unequilibrated cut in $(G,\sigma)$, a contradiction. Thus we assume that $x'y'\in C_i$ for each $i\in \{1,2\}$. Since exactly one of $xy$ and $xz$ is negative, for $i=1$, replacing $x'y'$ with $\{xy,xz\}$ results in an unequilibrated cut of $(G, \sigma)$; for $i=2$, adding $\{xy,xz\}$ to $C_2$ results in an unequilibrated cut of $(G, \sigma)$; both contradicting~\Cref{lem:bad-cut}. Thus we have that $$F(G_i, \sigma_i)= F(G, \sigma)-1 > \frac{1}{3} v(G)-1= \frac{1}{3}(v(G_i)+3)-1=\frac{1}{3}v(G_i).$$ 
 Now $v(G_1)=v(G_2)=v(G)-3\geq 7$. If $(G_i, \sigma_i)$ is not isomorphic to any of $\widehat{\Gamma}_i$'s, then it is a smaller counterexample, contradicting the minimality of $(G, \sigma)$. So we know that $v(G_i)=8$ and $(G_i, \sigma_i)$ can only be one of $\widehat{\Gamma}_3, \widehat{\Gamma}_4$, and $\widehat{\Gamma}_5$. Since $d_{G_2}(x')\le2$ but none of $\widehat{\Gamma}_3, \widehat{\Gamma}_4$, and $\widehat{\Gamma}_5$ has any vertex of degree less than $2$, we know that $i=1$. By~\Cref{col:equilibrated}, the newly-added positive edge $x'y'$ must be in an equilibrated cut of $(G_1,\sigma_1)$. Since exactly one of $xy$ and $xz$ is negative, and neither $xy$ nor $xz$ is in $(G_1, \sigma_1)$, replacing $x'y'$ with $\{xy,xz\}$ will imply an unequilibrated cut of $(G,\sigma)$, a contradiction~\Cref{lem:bad-cut}.
\end{proof}

\begin{lemma}\label{lem:33}
The minimum counterexample $(G, \sigma)$ contains no adjacent triangles. 
\end{lemma}

\begin{proof}
Assume to the contrary that there are two adjacent triangles $acd$ and $bcd$ in $G$. Let $x$ and $y$ denote the third neighbor of $a$ and $b$ respectively. Note that $x\ne y$, as otherwise the edge incident to $x$ which is different from $\{ax,bx\}$ is a cut-edge, a contradiction. Let $(H, \sigma)$ denote the signed subgraph of $(G, \sigma)$ induced by the vertex set $\{a,b,c,d\}$. 

We claim that there exists a signature $\pi$ (of $G$) switching equivalent to $\sigma$ such that  $|E^-_{(G, \pi)}|=F(G, \sigma)$ and $|E^-_{(H, \pi|_H)}|\leq 1$. If $|E^-_{(H, \sigma)}|\leq 1$, then we let $\pi=\sigma$; if $|E^-_{(H, \sigma)}|>1$, then by~\Cref{lem:bad-cut} either $E^-_{(H, \sigma)}=\{ac, bd\}$ or $E^-_{(H, \sigma)}=\{ad, bc\}$. Without loss of generality, assume that $E^-_{(H, \sigma)}=\{ac, bd\}$. We form $(G,\pi)$ by switching at the equilibrated cut $[a,d]$. Now $E^-_{(H, \pi|_H)}=\{cd\}$, and as we form $\pi$ from $\sigma$ by switching at an equilibrated cut, $|E^-_{(G, \pi)}|=F(G, \sigma)$, thus the claim is proved.

Now we form a signed graph $(G',\pi')$ from $(G, \pi)$ by deleting vertices $b,c,$ and $d$, and adding an edge $ay$ with $\pi'(ay)=\pi(by)$. 
Note that the resulting signed graph $(G',\pi')$ is still $2$-edge-connected. Furthermore, for any unequilibrated cut of $(G', \pi')$ containing the edge $ay$, replacing it with $by$ results in an unequilibrated cut in $(G,\pi)$. Therefore, $(G', \pi')$ contains no unequilibrated cut, and hence $F(G',\pi')=|E^-_{(G', \pi')}|$. Since $|E^-_{(H, \pi|_H)}|\leq 1$, we have that $$F(G', \pi')\geq F(G, \pi)-1>\frac{1}{3}v(G)-1= \frac{1}{3}(v(G')+3)-1= \frac{1}{3}v(G').$$ 
Note that $v(G')=v(G)-3\geq 7$. If $(G', \pi')$ is not isomorphic to any of $\widehat{\Gamma}_i$'s, then it is a smaller counterexample, contradicting the minimality of $(G, \pi)$. So we know that $v(G')=8$ and 
$(G', \pi')$ can only be one of $\widehat{\Gamma}_3, \widehat{\Gamma}_4$, and $\widehat{\Gamma}_5$. However, none of these three signed graphs contain any $2$-vertex, contradicting the fact that $a$ is a $2$-vertex in $(G', \pi')$.
\end{proof}

\begin{lemma}\label{lem:34}
The minimum counterexample $(G, \sigma)$ contains no triangle adjacent to a $4$-cycle, where the triangle has a negative edge that is not the common edge shared with the $4$-cycle.
\end{lemma}

\begin{figure}[htbp]
\begin{subfigure}[t]{.33\textwidth}
		\centering
	    \begin{tikzpicture}[>=latex,
		roundnode/.style={circle, draw=black!90, thick, minimum size=2mm, inner sep=0pt},
        squarenode/.style={rectangle, draw=black!90, thick, minimum size=2mm, inner sep=0pt},
        scale=0.7
		]
          \node[roundnode] (z) at (1,0.8) {};
          \node[roundnode] (c) at (0,1.6) {};
          \node[roundnode] (a) at (0,2.4) {};
          \node[roundnode] (x) at (0,3.2) {};
          \node[roundnode] (d) at (2,1.6) {};
          \node[roundnode] (b) at (2,2.4) {};
          \node[roundnode] (y) at (2,3.2) {};

          \draw[fill=white,line width=0.2pt] (1,0.8) node[right=1mm] {$z$};
          \draw[fill=white,line width=0.2pt] (0,1.6) node[left=1mm] {$c$};
          \draw[fill=white,line width=0.2pt] (0,2.4) node[left=1mm] {$a$};
          \draw[fill=white,line width=0.2pt] (0,3.2) node[left=1mm] {$x$};
          \draw[fill=white,line width=0.2pt] (2,1.6) node[right=1mm] {$d$};
          \draw[fill=white,line width=0.2pt] (2,2.4) node[right=1mm] {$b$};
          \draw[fill=white,line width=0.2pt] (2,3.2) node[right=1mm] {$y$};

          \draw[line width=1pt, gray] (c)--(a)--(x);
          \draw[line width=1pt, gray] (y)--(b)--(d);
          \draw[line width=1pt, gray] (a)--(b);
          \draw[line width=1pt, gray] (c)--(d);
          \draw[line width=1pt, blue]  (c)--(z);
          \draw[densely dotted, line width=1pt, red] (d)--(z);
	    \end{tikzpicture} 
        \caption{$(G,\sigma)$}  
        \label{fig:34f1}
  \end{subfigure}
  \begin{subfigure}[t]{.32\textwidth}
		\centering
	    \begin{tikzpicture}[>=latex,
		roundnode/.style={circle, draw=black!90, thick, minimum size=2mm, inner sep=0pt},
        squarenode/.style={rectangle, draw=black!90, thick, minimum size=2mm, inner sep=0pt},
        scale=0.7
		]

          \node[roundnode] (z) at (1,0.8) {};
          \node[roundnode] (c) at (0,1.6) {};
          \node[roundnode] (x) at (0,3.2) {};
          \node[roundnode] (y) at (2,3.2) {};

          \draw[fill=white,line width=0.2pt] (1,0.8) node[right=1mm] {$z$};
          \draw[fill=white,line width=0.2pt] (0,1.6) node[left=1mm] {$c$};
          \draw[fill=white,line width=0.2pt] (0,3.2) node[left=1mm] {$x$};
          \draw[fill=white,line width=0.2pt] (2,3.2) node[right=1mm] {$y$};

          \draw[line width=1pt, gray] (c)--(x);
          \draw[line width=1pt, gray] (z)--(y);
          \draw[line width=1pt, blue]  (c)--(z);
	    \end{tikzpicture} 
        \caption{$(G_1,\pi_1)$}  
        \label{fig:34f2}
        \end{subfigure}      
\begin{subfigure}[t]{.33\textwidth}
		\centering
	    \begin{tikzpicture}[>=latex,
		roundnode/.style={circle, draw=black!90, thick, minimum size=2mm, inner sep=0pt},
        squarenode/.style={rectangle, draw=black!90, thick, minimum size=2mm, inner sep=0pt},
        scale=0.7
		]
          \node[roundnode] (z) at (1,0.8) {};
          \node[roundnode] (x) at (0,3.2) {};
          \node[roundnode] (y) at (2,3.2) {};
          \node[roundnode] (u) at (2,1.6) {};

          \draw[fill=white,line width=0.2pt] (1,0.8) node[right=1mm] {$z$};
          \draw[fill=white,line width=0.2pt] (0,3.2) node[left=1mm] {$x$};
          \draw[fill=white,line width=0.2pt] (2,3.2) node[right=1mm] {$y$};
          \draw[fill=white,line width=0.2pt] (2,1.6) node[right=1mm] {$d$};

          \draw[line width=1pt, gray] (x)--(z);
          \draw[line width=1pt, gray] (y)--(u);
          \draw[line width=1pt, blue] (u)--(z);
	    \end{tikzpicture} 
        \caption{$(G_2,\pi_2)$}  
        \label{fig:34f4}
  \end{subfigure}

    \caption{Configurations in \Cref{lem:34}}  
    \label{fig:34}
\end{figure}  

\begin{proof}
    Assume not, and the underlying graph $G$ contains a triangle $cdz$ and a $4$-cycle $abdc$. Let $x$ and $y$ be the third neighbor of $a$ and $b$, respectively. Note that neither $x$ nor $y$ is coincide with $z$, as otherwise $by$ or $ax$ is a cut-edge, a contradiction. By symmetry we may assume that $dz$ is negative and thus $cz$ is positive. See~\Cref{fig:34f1}. Let $(H, \sigma)$ denote the signed subgraph of $(G, \sigma)$ induced by the vertex set $\{a,b,c,d,z\}$. 
    
    We claim that there exists a signature $\pi$ (of $G$) switching equivalent to $\sigma$ such that $|E^-_{(G, \pi)}|=F(G, \sigma)$ and $|E^-_{(H, \pi|_H)}|\leq 1$. If $|E^-_{(H, \sigma)}|\leq 1$, then we let $\pi=\sigma$; if $|E^-_{(H,\sigma)}|>1$, then by~\Cref{lem:bad-cut}, either $E^-_{(H,\sigma)}=\{ab,dz\}$ or $E^-_{(H,\sigma)}=\{ac,dz\}$. If $E^-_{(H,\sigma)}=\{ab,dz\}$, we form $(G,\pi)$ by switching at the equilibrated cut $[b,d]$; if $E^-_{(H,\sigma)}=\{ac,dz\}$, we form $(G,\pi)$ by switching at the equilibrated cut $[c,z]$. Now $E^-_{(H,\pi|_H)}=\{cd\}$, and as we form $\pi$ from $\sigma$ by switching at an equilibrated cut, $|E^-_{(G, \pi)}|=F(G, \sigma)$, thus the claim is proved.
    
    Since $z$ is a $3$-vertex, if both $x$ and $y$ are adjacent to $z$, then $x=y$. However, now we have that $v(G)=6$, contradicting the fact that $v(G)\geq 10$. So we consider the following two cases.
    
    \begin{itemize}
    \setlength{\itemsep}{0em}
        \item If $y$ is not adjacent to $z$, then let $(G_1,\pi_1)$ be formed from $(G,\pi)$ by deleting vertices $a,b$, and $d$, and adding edges $xc$ and $yz$ such that $\pi_1(xc)=\pi(xa)$ and $\pi_1(yz)=\pi(yb)$, as shown in~\Cref{fig:34f2};
        \item If $x$ is not adjacent to $z$, then let  $(G_2,\pi_2)$ be formed from $(G,\pi)$ by deleting vertices $a,b$, and $c$, adding edges $yd$ and $xz$ such that $\pi_2(yd)=\pi(yb)$ and $\pi_2(xz)=\pi(xa)$, and changing the sign of $dz$ to be positive under $\pi_2$, as shown in~\Cref{fig:34f4}.
    \end{itemize}
    It is easy to see that both $G_1$ and $G_2$ are $2$-edge-connected. 

    Next we shall prove that $|E^-_{(G_i, \pi_i)}|=F(G_i, \pi_i)$ for each $i\in \{1,2\}$. Assume to the contrary that $|E^-_{(G_i, \pi_i)}|>F(G_i, \pi_i)$. It implies that $(G_i,\pi_i)$ contains an unequilibrated cut. We claim that there exists an unequilibrated cut of $(G_1,\pi_1)$ not containing $cz$. For a cut containing $cz$, if it does not contain $cx$, then replacing $cz$ with $cx$ will result in an unequilibrated cut of $(G_1, \pi_1)$; if it contains $cx$, then deleting $cx$ and $cz$ will result in an unequilibrated cut $(G_1, \pi_1)$. Thus we shall choose the one not containing $cz$ and denote it by $C_1:=[A_1,B_1]$. 
    Now based on $C_1$, we consider the following cut $C_1'$ of $(G, \pi)$ as follows:
    \begin{equation*}
	C_1'=\left\lbrace  
	\begin{array}{ll}
    [A_1,B_1\cup\{a,b,d\}] &\text{ if } x\in A_1,\{y,z\}\subset B_1,\\
        
      [A_1,B_1\cup\{a,b,d\}] &\text{ if }\{x,y\}\subset A_1,z\in B_1,\\
        
		[A_1\cup\{a,b,d\},B_1] &\text{ if }\{x,z\}\subset A_1,y\in B_1,\\

          [A_1\cup\{a,b,d\},B_1]&\text{ if }\{x,y,z\}\subset A_1,
	\end{array}
	\right.
    \end{equation*}
where $C_1'=C_1\setminus \{cx\}\cup\{ax\}$ in the first case, $C_1'=C_1\setminus \{cx,zy\}\cup\{ax,by\}$ in the second case, $C_1'=C_1\setminus \{zy\}\cup\{by\}$ in the third case, and $C_1'=C_1$ in the last case.
Now $C_1'$ is an unequilibrated cut in $(G,\pi)$, a contradiction to~\Cref{lem:bad-cut}. By symmetry, there exists an unequilibrated cut $C_2$ of $(G_2,\pi_2)$ not containing $dz$, and based on $C_2$ we shall again find an unequilibrated cut in $(G, \pi)$. So $|E^-_{(G_i, \pi_i)}|=F(G_i, \pi_i)$ for each $i\in \{1,2\}$, and in both cases, as $|E^-_{(H, \pi|_H)}|\leq 1$, we have that $$F(G_i, \pi_i)\geq F(G, \pi)-1>\frac{1}{3}v(G)-1= \frac{1}{3}(v(G_i)+3)-1= \frac{1}{3}v(G_i).$$ Note that $v(G_i)=v(G)-3\geq 7$. If $(G_i, \pi_i)$ is not isomorphic to any of $\widehat{\Gamma}_i$'s, then it is a smaller counterexample, contradicting the minimality of $(G, \pi)$. So we know that $v(G_i)=8$ and $(G_i, \pi_i)$ can only be one of $\widehat{\Gamma}_3, \widehat{\Gamma}_4$, and $\widehat{\Gamma}_5$. However, $G_1$ has a $2$-vertex $c$ and $G_2$ has a $2$-vertex $d$, contradicting the fact that none of $\widehat{\Gamma}_3, \widehat{\Gamma}_4$, and $\widehat{\Gamma}_5$ has any $2$-vertex.
\end{proof}

\begin{lemma}\label{lem:4435}
The minimum counterexample $(G, \sigma)$ does not contain the following: 
\begin{enumerate}[label=(\arabic*)]
    \itemsep 0em
    \item\label{lem:44} Adjacent negative $4$-cycles sharing a common negative edge, such that the resulting $6$-cycle has exactly one chord;
    \item\label{lem:35} Adjacent negative $3$-cycle and $5$-cycle sharing a common negative edge, such that the resulting $6$-cycle has exactly one chord.
\end{enumerate}      
\end{lemma}

\begin{figure}[htbp]
  \begin{subfigure}[t]{.25\textwidth}
		\centering
	    \begin{tikzpicture}[>=latex,
		roundnode/.style={circle, draw=black!90, thick, minimum size=2mm, inner sep=0pt},
        squarenode/.style={rectangle, draw=black!90, thick, minimum size=2mm, inner sep=0pt},
        scale=0.6
		]
          \node[roundnode] (a) at (0,0) {};
          \node[roundnode] (a') at (-1,0) {};
          \node[roundnode] (b) at (0,-1.5) {};
          \node[roundnode] (b') at (-1,-1.5) {};
          \node[roundnode] (x) at (1.5,0) {};
          \node[roundnode] (y) at (1.5,-1.5) {};
          \node[roundnode] (c) at (3,0) {};
          \node[roundnode] (c') at (4,0) {};
          \node[roundnode] (d) at (3,-1.5) {};
          \node[roundnode] (d') at (4,-1.5) {};

          \draw[fill=white,line width=0.2pt] (0,0) node[above=1mm] {$a$};
          \draw[fill=white,line width=0.2pt] (-1,0) node[above=1mm] {$a'$};
          \draw[fill=white,line width=0.2pt] (0,-1.5) node[below=1mm] {$b$};
          \draw[fill=white,line width=0.2pt] (-1,-1.5) node[below=1mm] {$b'$};
          \draw[fill=white,line width=0.2pt] (1.5,0) node[above=1mm] {$x$};
          \draw[fill=white,line width=0.2pt] (1.5,-1.5) node[below=1mm] {$y$};
          \draw[fill=white,line width=0.2pt] (3,0) node[above=1mm] {$c$};
          \draw[fill=white,line width=0.2pt] (4,0) node[above=1mm] {$c'$};
          \draw[fill=white,line width=0.2pt] (3,-1.5) node[below=1mm] {$d$};
          \draw[fill=white,line width=0.2pt] (4,-1.5) node[below=1mm] {$d'$};
          \draw[line width=1pt, blue] (a)--(x)--(c)--(d)--(y)--(b)--(a);
          \draw[line width=1pt, gray] (a)--(a');
          \draw[line width=1pt, gray] (b)--(b');
          \draw[line width=1pt, gray] (c)--(c');
          \draw[line width=1pt, gray] (d)--(d');
          \draw[densely dotted, line width=1pt, red] (x)--(y);
	    \end{tikzpicture} 
        \caption{Adjacent $4$-cycles}  
        \label{fig:C4-C4-switching}
  \end{subfigure}
  \begin{subfigure}[t]{.25\textwidth}
		\centering
	    \begin{tikzpicture}[>=latex,
		roundnode/.style={circle, draw=black!90, thick, minimum size=2mm, inner sep=0pt},
        squarenode/.style={rectangle, draw=black!90, thick, minimum size=2mm, inner sep=0pt},
        scale=0.6
		]
          \node[roundnode] (a) at (0,0) {};
          \node[roundnode] (a') at (-1,0) {};
          \node[roundnode] (b) at (0,-1.5) {};
          \node[roundnode] (b') at (-1,-1.5) {};
          \node[roundnode] (x) at (1.5,0) {};
          \node[roundnode] (y) at (3,-0.75) {};
          \node[roundnode] (c) at (3,0) {};
          \node[roundnode] (c') at (4,0) {};
          \node[roundnode] (d) at (3,-1.5) {};
          \node[roundnode] (d') at (4,-1.5) {};

          \draw[fill=white,line width=0.2pt] (0,0) node[above=1mm] {$a$};
          \draw[fill=white,line width=0.2pt] (-1,0) node[above=1mm] {$a'$};
          \draw[fill=white,line width=0.2pt] (0,-1.5) node[below=1mm] {$b$};
          \draw[fill=white,line width=0.2pt] (-1,-1.5) node[below=1mm] {$b'$};
          \draw[fill=white,line width=0.2pt] (1.5,0) node[above=1mm] {$x$};
          \draw[fill=white,line width=0.2pt] (3,-0.75) node[right=1mm] {$y$};
          \draw[fill=white,line width=0.2pt] (3,0) node[above=1mm] {$c$};
          \draw[fill=white,line width=0.2pt] (4,0) node[above=1mm] {$c'$};
          \draw[fill=white,line width=0.2pt] (3,-1.5) node[below=1mm] {$d$};
          \draw[fill=white,line width=0.2pt] (4,-1.5) node[below=1mm] {$d'$};
          \draw[line width=1pt, blue] (a)--(x)--(c)--(y)--(d)--(b)--(a);
          \draw[line width=1pt, gray] (a)--(a');
          \draw[line width=1pt, gray] (b)--(b');
          \draw[line width=1pt, gray] (c)--(c');
          \draw[line width=1pt, gray] (d)--(d');
          \draw[densely dotted, line width=1pt, red] (x)--(y);
	    \end{tikzpicture} 
        \caption{Adjacent $3$-cycle and $5$-cycle}  
        \label{fig:C3-C5-switching}
        \end{subfigure}
  \begin{subfigure}[t]{.24\textwidth}
  \centering
		\begin{tikzpicture}[>=latex,
		roundnode/.style={circle, draw=black!90, thick, minimum size=2mm, inner sep=0pt},
        squarenode/.style={rectangle, draw=black!90, thick, minimum size=2mm, inner sep=0pt},
        scale=0.6
		]
          \node[roundnode] (a') at (-1,0) {};
          \node[roundnode] (b') at (-1,-1.5) {};
          \node[roundnode] (x) at (1,0) {};
          \node[roundnode] (y) at (1,-1.5) {};
          \node[roundnode] (c') at (3,0) {};
          \node[roundnode] (d') at (3,-1.5) {};

          \draw[fill=white,line width=0.2pt] (-1,0) node[above=1mm] {$a'$};
          \draw[fill=white,line width=0.2pt] (-1,-1.5) node[below=1mm] {$b'$};
          \draw[fill=white,line width=0.2pt] (1,0) node[above=1mm] {$u$};
          \draw[fill=white,line width=0.2pt] (1,-1.5) node[below=1mm] {$v$};
          \draw[fill=white,line width=0.2pt] (3,0) node[above=1mm] {$c'$};
          \draw[fill=white,line width=0.2pt] (3,-1.5) node[below=1mm] {$d'$};
          
          \draw[line width=1pt, gray] (a')--(x)--(c');
          \draw[line width=1pt, gray] (b')--(y)--(d');
          \draw[line width=1pt, blue] (x)--(y);
	    \end{tikzpicture} 
        \caption{$(G_1, \sigma_1)$}
		\label{fig:C4-C4modified1}
    \end{subfigure}
      \begin{subfigure}[t]{.24\textwidth}
  \centering
		\begin{tikzpicture}[>=latex,
		roundnode/.style={circle, draw=black!90, thick, minimum size=2mm, inner sep=0pt},
        squarenode/.style={rectangle, draw=black!90, thick, minimum size=2mm, inner sep=0pt},
        scale=0.6
		]
          \node[roundnode] (a') at (-1,0) {};
          \node[roundnode] (b') at (-1,-1.5) {};
          \node[roundnode] (u) at (0.5,-0.75) {};
          \node[roundnode] (v) at (2,-0.75) {};
          \node[roundnode] (c') at (3.5,0) {};
          \node[roundnode] (d') at (3.5,-1.5) {};

          \draw[fill=white,line width=0.2pt] (-1,0) node[above=1mm] {$a'$};
          \draw[fill=white,line width=0.2pt] (-1,-1.5) node[below=1mm] {$b'$};
          \draw[fill=white,line width=0.2pt] (0.5,-0.75) node[above=1mm] {$u$};
          \draw[fill=white,line width=0.2pt] (2,-0.75) node[above=1mm] {$v$};
          \draw[fill=white,line width=0.2pt] (3.5,0) node[above=1mm] {$c'$};
          \draw[fill=white,line width=0.2pt] (3.5,-1.5) node[below=1mm] {$d'$};
          
          \draw[line width=1pt, gray] (a')--(u)--(b');
          \draw[line width=1pt, gray] (c')--(v)--(d');
          \draw[line width=1pt, blue] (u)--(v);
	    \end{tikzpicture} 
        \caption{$(G_2, \sigma_2)$}
		\label{fig:C4-C4modified2}
    \end{subfigure}
    \caption{Configurations in \Cref{lem:4435}}  
    \label{fig:C4-C4-all}
\end{figure}

\begin{proof}
Assume to the contrary that $(G, \sigma)$ contains two negative $4$-cycles $abyx$ and $cdyx$ sharing one common negative edge $xy$, as shown in~\Cref{fig:C4-C4-switching}; or contains a $3$-cycle $cxy$ and a $5$-cycle $abdyx$ sharing one common negative edge $xy$, as shown in~\Cref{fig:C3-C5-switching}. For both cases, let $a',b',c',$ and $d'$ be the third neighbors of vertices $a,b,c,$ and $d$, respectively. Since the $6$-cycle $axcdyb$ or $axcydb$ has only one chord $xy$, $\{a,b\}\cap ( N(c)\cup N(d))=\emptyset$.
For both cases, we construct two signed graphs $(G_1, \sigma_1)$ and $(G_2, \sigma_2)$ as follows: delete vertices $\{a,b,c,d,x,y\}$ and replace the part connecting $a',b',c'$, and $d'$ by the two signed subgraphs depicted in~\Cref{fig:C4-C4modified1} and \Cref{fig:C4-C4modified2}. In particular, we define that $\sigma_1(a'u)=\sigma_2(a'u)=\sigma(a'a)$, $\sigma_1(b'v)=\sigma_2(b'u)=\sigma(b'b)$, $\sigma_1(c'u)=\sigma_2(c'v)=\sigma(c'c)$, and $\sigma_1(d'v)=\sigma_2(d'v)=\sigma(d'd)$. The resulting signed graphs are denoted by $(G_1, \sigma_1)$ and $(G_2, \sigma_2)$, respectively.

We first claim that at least one of $(G_1, \sigma_1)$ and $(G_2, \sigma_2)$ is $2$-edge-connected and has no parallel edges, and denote it by $(G', \sigma')$. We consider the vertex set $\{a',b',c',d'\}$. If three of them coincide, say $a'=b'=c'$, then $dd'$ is a cut-edge of $G$, a contradiction; If two of them coincide, then for cases that $a'=d'$ or $b'=c'$, both graphs are $2$-edge-connected with no parallel edges, and for other cases, by symmetry we only consider the cases that $a'=c'$ or $a'=b'$. For the case $a'=c'$, $(G_2,\sigma_2)$ is $2$-edge-connected with no parallel edges, and for the case $a'=b'$, $(G_1,\sigma_1)$ is $2$-edge-connected with no parallel edges. If all  $a',b',c',$ and $d'$ are distinct, then both $(G_1,\sigma_1)$ and $(G_2,\sigma_2)$ has no parallel edges, and as $G$ is $2$-edge-connected, at least one of them is $2$-edge-connected. Thus the claim is proved.

Next we claim that $F(G_i,\sigma_i)=F(G,\sigma)-1$ for each $i\in \{1,2\}$. Assume not and thus $(G',\sigma')$ contains an unequilibrated cut $C$. We form a new cut $C'$ of $(G, \sigma)$ based on the following edge replacement operations. 

(i) If $C$ contains $uv$, then: for the case of two adjacent negative $4$-cycles, we replace $uv$ by $\{ab,xy,cd\}$ when $i=1$ and replace $uv$ by $\{cx,xy,by\}$ when $i=2$; for the case of adjacent negative $3$-cycle and $5$-cycle, we replace $uv$ by $\{ab,xy,cy\}$ when $i=1$ and replace $uv$ by $\{cx,xy,bd\}$ when $i=2$.

(ii) If $C$ contains $wz$ for $i=1,2$, where $w\in\{a',b',c',d'\}$ and $z\in\{u,v\}$, then for both cases, we replace $wz$ by $ww_0$, where $w_0$ is the vertex that satisfies the following conditions: $w_0\in N(w)$, $w_0$ is the vertex which is contained in the $6$-cycle, and $\sigma(ww_0)=\sigma(wz)$; 

(iii) Otherwise, let $C'$=$C$.

Note that the conditions (i) and (ii) can occur simultaneously. If so, then we apply both operations.
It is easy to see that $C'$ is an unequilibrated cut in $(G,\sigma)$, a contradiction to~\Cref{lem:bad-cut}. Hence, $$F(G',\sigma')=F(G, \sigma)-1>\frac{1}{3}v(G)-1=\frac{1}{3}(v(G')+4)-1> \frac{1}{3}v(G').$$ 
Note that $v(G')=v(G)-4\geq 6$. If $(G', \sigma')$ is not isomorphic to any of $\widehat{\Gamma}_i$'s, then it is a smaller counterexample, contradicting the minimality of $(G, \sigma)$. So it must hold that $v(G')=8$ and 
$(G', \sigma')$ can only be one of $\widehat{\Gamma}_3, \widehat{\Gamma}_4$, and $\widehat{\Gamma}_5$. Note that as $v(G)=v(G')+4=12$, $(G,\sigma)$ is a signed graph on $12$ vertices with $4$ negative edges and thus $F(G, \sigma)=\frac{1}{3}v(G)$, a contradiction.
\end{proof}

\begin{lemma}\label{lem:2edgecut}
Assume that $\{v_1u_1, v_2u_2\}$ is a $2$-edge-cut of the minimum counterexample $(G, \sigma)$. Let $H$ be a $2$-edge-connected component of $G-\{v_1u_1, v_2u_2\}$, where $\{u_1,u_2\}\subset V(H)$. If both $v_1u_1$ and $v_2u_2$ are positive, and $v_1$ is not adjacent to $v_2$, then neither $v_1$ nor $v_2$ is incident with a negative edge. 
\end{lemma}

\begin{figure}[htbp]
  \begin{subfigure}[t]{.48\textwidth}
		\centering
	    \begin{tikzpicture}[>=latex,
		roundnode/.style={circle, draw=black!90, thick, minimum size=2mm, inner sep=0pt},
        squarenode/.style={rectangle, draw=black!90, thick, minimum size=2mm, inner sep=0pt},
        scale=0.6
		]

          \node[roundnode] (u1) at (0,1) {};
          \node[roundnode] (u2) at (3,1) {};
          \node[roundnode] (v1) at (-2,1) {};
          \node[roundnode] (v2) at (5,1) {};
          \node[roundnode] (x) at (-3,2) {};
          \node[roundnode] (y) at (-3,0) {};
          
          \draw[fill=white,line width=0.2pt] (0,1) node[right=1mm] {$u_1$};
          \draw[fill=white,line width=0.2pt] (3,1) node[left=1mm] {$u_2$};
          \draw[fill=white,line width=0.2pt] (-2,1) node[above=1mm] {$v_1$};
          \draw[fill=white,line width=0.2pt] (5,1) node[above=1mm] {$v_2$};
          \draw[fill=white,line width=0.2pt] (-3,2) node[left=1mm] {$x$};
          \draw[fill=white,line width=0.2pt] (-3,0) node[left=1mm] {$y$};

          \draw[line width=1pt, gray] (0,0)--(u1)--(0,2)--(3,2)--(u2)--(3,0)--(0,0);
          \draw[line width=1pt, blue] (u1)--(v1)--(y);
          \draw[line width=1pt, blue] (u2)--(v2);
          \draw[densely dotted, line width=1pt, red] (x)--(v1);
	    \end{tikzpicture} 
        \caption{$(G,\sigma)$}  
        \label{fig:2edgecut1}
  \end{subfigure}
  \begin{subfigure}[t]{.48\textwidth}
		\centering
	    \begin{tikzpicture}[>=latex,
		roundnode/.style={circle, draw=black!90, thick, minimum size=2mm, inner sep=0pt},
        squarenode/.style={rectangle, draw=black!90, thick, minimum size=2mm, inner sep=0pt},
        scale=0.6
		]

          \node[roundnode] (v1) at (-2,1) {};
          \node[roundnode] (v2) at (3,1) {};
          \node[roundnode] (x) at (-3,2) {};
          \node[roundnode] (y) at (-3,0) {};

          \draw[fill=white,line width=0.2pt] (-2,1) node[above=1mm] {$v_1$};
          \draw[fill=white,line width=0.2pt] (3,1) node[above=1mm] {$v_2$};
          \draw[fill=white,line width=0.2pt] (-3,2) node[left=1mm] {$x$};
          \draw[fill=white,line width=0.2pt] (-3,0) node[left=1mm] {$y$};

          \draw[line width=1pt, blue] (v2)--(v1)--(y);
          \draw[densely dotted, line width=1pt, red] (x)--(v1);
	    \end{tikzpicture} 
        \caption{$(G',\sigma')$}  
        \label{fig:2edgecut2}
  \end{subfigure}
    \caption{Configurations in \Cref{lem:2edgecut}}  
    \label{fig:2edgecut}
\end{figure}  

\begin{proof}
Let $\{u_1v_1,u_2v_2\}:=[V(H),V(G)\setminus V(H)]$ be a $2$-edge-cut of $(G,\sigma)$ where $\{u_1,u_2\}\subset V(H)$. Without loss of generality, assume to the contrary that $v_1$ is incident to a negative edge $v_1x$, see~\Cref{fig:2edgecut1}. Note that $v_1 \ne v_2$, since otherwise $G$ would contain either a cut-edge that connects $v_1$, or a $2$-vertex that is $v_1$ itself, either of which leads to a contradiction. Note that $H$ is $2$-edge-connected and has no parallel edges.

We first claim that $|E^-_{(H, \sigma|_H)}|=F(H, \sigma|_H)$. Assume to the contrary that $|E^-_{(H, \sigma|_H)}|>F(H, \sigma|_H)$, so there exists an unequilibrated cut $C:=[H_1,H_2]$ where $V(H)=H_1\cup H_2$ in $(H, \sigma|_H)$. Note that $C$ must separate $u_1$ and $u_2$, as otherwise $C$ itself is an unequilibrated cut of $(G,\sigma)$, a contradiction. Without loss of generality, we assume that $u_1\in H_1$ and $u_2\in H_2$. Now since $v_1\notin H$, the cut $[H_1\cup\{v_1\},V(G)\setminus(H_1\cup\{v_1\})]$, which is indeed $C\cup\{xv_1,yv_1\}$, is an unequilibrated cut of $(G,\sigma)$, a contradiction to~\Cref{lem:bad-cut}. We next claim that $F(H, \sigma|_H)\le \frac{1}{3}v(H)$. Assume not and since 
$(H, \sigma|_H)$ is $2$-edge-connected, and has no parallel edges, by the minimality of the choice of $(G, \sigma)$, $(H, \sigma|_H)$ is isomorphic to one of $\widehat{\Gamma}_i$'s. But it is not possible since $u_1$ and $u_2$ are two $2$-vertices in $(H, \sigma|_H)$ but none of the exceptional graphs has two $2$-vertices.

Now we form a signed graph $(G',\sigma')$ from $(G, \sigma)$ as follows: delete all vertices of $H$, and add a positive edge $v_1v_2$.  Note that $(G',\sigma')$ is $2$-edge-connected and has no parallel edges. 
Furthermore, for any unequilibrated cut of $(G', \sigma')$ containing the edge $v_1v_2$, replacing it with $u_2v_2$ results in an unequilibrated cut in $(G,\sigma)$. Therefore, $(G', \sigma')$ contains no unequilibrated cut, and hence $F(G',\sigma')=|E^-_{(G', \sigma')}|$ where $|E^-_{(G',\sigma')}|=F(G, \sigma)-F(H, \sigma|_H)$. Since $F(H, \sigma|_H)\le \frac{1}{3}v(H)$, we have that $$F(G', \sigma')=F(G, \sigma)-F(H, \sigma|_H)>\frac{1}{3}v(G)-\frac{1}{3}v(H)=  \frac{1}{3}v(G').$$ 
If $(G', \sigma')$ is not isomorphic to any of $\widehat{\Gamma}_i$'s, then it is a smaller counterexample, contradicting the minimality of $(G, \sigma)$. So $(G', \sigma')$ must be one of $\widehat{\Gamma}_i$'s. 

We discuss these five signed graphs separately. Note that $(G',\sigma')-v_1v_2$ is also a subgraph of $(G,\sigma)$. If $(G',\sigma')$ is isomorphic to either $\widehat{\Gamma}_1$ or $\widehat{\Gamma}_5$, then as both $\widehat{\Gamma}_1$ and $\widehat{\Gamma}_5$ are edge-transitive, there exists $\sigma_1$ switching-equivalent to $\sigma$ with $|E^-_{(G,\sigma_1)}|=F(G,\sigma)$ satisfying the following conditions: (1) either $(G',\sigma_1)-v_1v_2$ contains a pair of adjacent triangles sharing a negative edge, which contradicts~\Cref{lem:33}, (2) or $(G',\sigma_1)-v_1v_2$ contains adjacent $4$-cycles where the common edge is negative, and the $6$-cycle they formed has only one chord, which contradicts~\Cref{lem:4435}\ref{lem:44}.
If $(G',\sigma')$ is isomorphic to $\widehat{\Gamma}_2$, then $G$ has a $2$-vertex, a contradiction to~\Cref{lem:cubic}. If $(G',\sigma')$ is isomorphic to $\widehat{\Gamma}_3$, then $(G',\sigma')-v_1v_2$ must contain two adjacent triangles, a contradiction to~\Cref{lem:33}. If $(G',\sigma')$ is isomorphic to $\widehat{\Gamma}_4$, then as $\widehat{\Gamma}_4$ is a signed planar graph and every facial cycle is negative, there exists $\sigma^*$ switching-equivalent to $\sigma'$ with $|E^-_{(G,\sigma^*)}|=F(G', \sigma')$ satisfying the following: (1) $(G', \sigma^*)-v_1v_2$ contains a triangle adjacent to a $4$-cycle, where one edge of the $3$-cycle that is not the common edge shared with the $4$-cycle is negative, which contradicts~\Cref{lem:34}; or (2) $(G',\sigma^*)-v_1v_2$ contains a triangle adjacent to a $5$-cycle, where the common shared edge is negative, and the $6$-cycle they formed has only one chord, which contradicts~\Cref{lem:4435}\ref{lem:35}.
\end{proof}

\begin{lemma}\label{lem:c3negative}
The minimum counterexample $(G, \sigma)$ contains no negative triangle. 
\end{lemma}

\begin{figure}[htbp]
\begin{subfigure}[t]{.32\textwidth}
		\centering
	    \begin{tikzpicture}[>=latex,
		roundnode/.style={circle, draw=black!90, thick, minimum size=2mm, inner sep=0pt},
        squarenode/.style={rectangle, draw=black!90, thick, minimum size=2mm, inner sep=0pt},
        scale=0.7
		]
          \node[roundnode] (a) at (0.5,2) {};
          \node[roundnode] (b) at (0,1) {};
          \node[roundnode] (c) at (1,1) {};
          \node[roundnode] (x) at (0.5,3) {};
          \node[roundnode] (y) at (-1,1) {};
          \node[roundnode] (z) at (2,1) {};

          \draw[fill=white,line width=0.2pt] (0.5,2) node[left=1mm] {$a$};
          \draw[fill=white,line width=0.2pt] (0,1) node[below=0.5mm] {$b$};
          \draw[fill=white,line width=0.2pt] (1,1) node[below=1.2mm] {$c$};
          \draw[fill=white,line width=0.2pt] (0.5,3) node[above=1mm] {$x$};
          \draw[fill=white,line width=0.2pt] (-1,1) node[left=1mm] {$y$};
          \draw[fill=white,line width=0.2pt] (2,1) node[right=0.53mm] {$z$};

          \draw[line width=1pt, blue] (a)--(b)--(y);
          \draw[line width=1pt, blue] (a)--(c)--(z);
          \draw[line width=1pt, gray] (x)--(a);
          \draw[line width=1pt, blue] (c)--(z);
          \draw[densely dotted, line width=1pt, red] (b)--(c);
	    \end{tikzpicture} 
        \caption{$(G,\sigma)$}  
        \label{fig:c3f1}
  \end{subfigure}
  \begin{subfigure}[t]{.33\textwidth}
		\centering
	    \begin{tikzpicture}[>=latex,
		roundnode/.style={circle, draw=black!90, thick, minimum size=2mm, inner sep=0pt},
        squarenode/.style={rectangle, draw=black!90, thick, minimum size=2mm, inner sep=0pt},
        scale=0.7
		]
          \node[roundnode] (a) at (0.5,2) {};
          \node[roundnode] (x) at (0.5,3) {};
          \node[roundnode] (y) at (-1,1) {};
          \node[roundnode] (z) at (2,1) {};

          \draw[fill=white,line width=0.2pt] (0.5,2) node[left=1mm] {$a$};
          \draw[fill=white,line width=0.2pt] (0.5,3) node[above=1mm] {$x$};
          \draw[fill=white,line width=0.2pt] (-1,1) node[left=1mm] {$y$};
          \draw[fill=white,line width=0.2pt] (2,1) node[right=0.5mm] {$z$};

          \draw[line width=1pt, gray] (x)--(a);
          \draw[line width=1pt, blue] (a)--(y);
          \draw[line width=1pt, blue] (a)--(z);
          \draw[white,line width=0.2pt] (0,1) node[white, below=0.5mm] {$b$};
          \draw[white,line width=0.2pt] (1,1) node[white, below=1.2mm] {$c$};
	    \end{tikzpicture} 
        \caption{$(G',\sigma')$}  
        \label{fig:c3f2}
  \end{subfigure}
    \begin{subfigure}[t]{.33\textwidth}
		\centering
	    \begin{tikzpicture}[>=latex,
		roundnode/.style={circle, draw=black!90, thick, minimum size=2mm, inner sep=0pt},
        squarenode/.style={rectangle, draw=black!90, thick, minimum size=2mm, inner sep=0pt},
        scale=0.7
		]
          \node[roundnode] (x) at (0.5,3) {};
          \node[roundnode] (y) at (-1,1) {};
          \node[roundnode] (z) at (2,1) {};
          \draw[fill=white,line width=0.2pt] (0.5,3) node[above=1mm] {$x$};
          \draw[fill=white,line width=0.2pt] (-1,1) node[left=0.5mm] {$y$};
          \draw[fill=white,line width=0.2pt] (2,1) node[right=0.53mm] {$z$};
          \draw[line width=1pt, gray] (x)--(z);

          \draw[white,line width=0.2pt] (0,1) node[white, below=0.5mm] {$b$};
          \draw[white,line width=0.2pt] (1,1) node[white, below=1.2mm] {$c$};
	    \end{tikzpicture} 
        \caption{$(G'',\sigma'')$}  
        \label{fig:c3f4}
  \end{subfigure}
    \caption{Configurations in \Cref{lem:c3negative}}  
    \label{fig:c3negative}
\end{figure}

\begin{proof}
    Assume to the contrary that $(G, \sigma)$ contains a negative triangle $abc$ where $bc$ is negative. Let $x,y,$ and $z$ be the third neighbors of $a,b,$ and $c$, respectively. By~\Cref{lem:33} we know that vertices $x,y,z$ are distinct, and furthermore, by~\Cref{lem:34} $x$ is not the neighbor of $y$ or $z$, as shown in~\Cref{fig:c3f1}. We form a signed graph $(G',\sigma')$ from $(G,\sigma)$ by deleting vertices $b$ and $c$, and adding two positive edges $ay$ and $az$, as shown in~\Cref{fig:c3f2}. Note that $(G',\sigma')$ is also $2$-edge-connected. Now we consider the following two cases based on the situation of the edge $ya$:
    
\smallskip
\noindent
{\bf Case 1.} \emph{The edge $ya$ is in a $2$-edge-cut of $(G',\sigma')$.}
\smallskip

In this case, we may assume that $\{ya,uv\}:=[V(H), V(G')\setminus V(H)]$ is a $2$-edge-cut in $(G',\sigma')$ where $\{y,u\}\subset H$ and $H$ is $2$-edge-connected. It is easy to see that $[V(H), V(G)\setminus V(H)]$(=$\{yb,uv\}$) is an edge-cut in $(G,\sigma)$ where $\{y,u\}\subset H$ and $H$ is $2$-edge-connected. We first claim that $\{a,b,c\}\cap V(H)=\emptyset$. As $yb$ is in the $2$-edge-cut, $b\notin V(H)$. Assume to the contrary that either: $a\in V(H)$ and $c\notin V(H)$, or $a\notin V(H)$ and $c\in V(H)$, or $\{a,c\}\subset V(H)$. Then $[V(H), V(G)\setminus V(H)]$ contains: $\{yb,ab,ac\}$, $\{yb,ac,bc\}$, or $\{yb,ab,bc\}$ respectively, each contradicting the fact that $[V(H), V(G)\setminus V(H)]$ is a $2$-edge-cut. We then claim that $uv$ is a positive edge, as otherwise $[V(H)\cup \{b\},V(G)\setminus (V(H)\cup\{b\})]:=\{uv,ab,bc\}$ is an unequilibrated cut of $(G, \sigma)$, a contradiction to~\Cref{lem:bad-cut}. We finally claim that $v$ is not adjacent to $b$. Assume not, and since $v\notin V(H)$ and $y\in V(H)$, we have $v\in \{a,c\}$. As $\{a,b,c\}\cap V(H)=\emptyset$, if $v=a$, then $uv=ax$, and $[V(H)\cup\{a,b,c\}, V(G)\setminus (V(H)\cup\{a,b,c\})]=\{cz\}$; if $v=c$, then $uv=cz$, and $[V(H)\cup\{a,b,c\}, V(G)\setminus (V(H)\cup\{a,b,c\})]=\{ax\}$. In both cases $G$ has a cut-edge, a contradiction. So now there exists an all-positive edge-cut $\{yb,uv\}$ such $b$ is incident to a negative edge $bc$, and not adjacent to $v$, a contradiction to~\Cref{lem:2edgecut}.

\smallskip
\noindent
{\bf Case 2.} \emph{The edge $ya$ is not in any $2$-edge-cut of $(G', \sigma')$.}
\smallskip

In this case, we form $(G'',\sigma'')$ from $(G',\sigma')$ by deleting the vertex $a$ and adding an edge $xz$ where $\sigma''(xz)=\sigma'(xa)$, as shown in~\Cref{fig:c3f4}. Since $x$ is not a neighbor of $z$, $(G'',\sigma'')$ has no parallel edges; and as $ya$ is not in any $2$-edge-cut of $(G', \sigma')$, $(G'', \sigma'')$ is $2$-edge-connected. Next we shall prove that $F(G'', \sigma'')=|E^-_{(G'', \sigma'')}|$. Assume not and it implies that $(G'', \sigma'')$ contains an unequilibrated cut $C:=[A,B]$. We consider the following cut $C'$ of $(G, \sigma)$:
\begin{equation*}
	C'=\left\lbrace  
	\begin{array}{ll}
[A,B\cup\{a,b,c\}]&\text{ if }x\in A,\{y,z\}\subset B\\
        
       [A\cup\{b\},B\cup\{a,c\}] &\text{ if }\{x,y\}\subset A,z\in B,\\
        
		[A\cup\{a,c\},B\cup \{b\}] &\text{ if } \{x,z\}\subset A,y\in B,\\

        [A\cup\{a,b,c\},B]&\text{ if } \{x,y,z\}\subset A,
	\end{array}
	\right.
    \end{equation*}
    where $C'=C\setminus \{xz\}\cup \{xa\}$ in the first case, $C'=C\setminus \{xz\} \cup \{xa,ab,bc\}$ in the second case, $C'=C\cup \{ab,bc\}$ in the third case, and $C'=C$ in the last case. In each case, $C'$ is an unequilibrated cut in $(G, \sigma)$, a contradiction to~\Cref{lem:bad-cut}. Thus we have that $$F(G'', \sigma'')=F(G, \sigma)-1>\frac{1}{3}v(G)-1= \frac{1}{3}(v(G'')+3)-1= \frac{1}{3}v(G'').$$ Note that $v(G'')=v(G)-3\geq 7$. If $(G'', \sigma'')$ is not isomorphic to any of $\widehat{\Gamma}_i$'s, then it is a smaller counterexample, contradicting the minimality of $(G, \sigma)$. So we know that $v(G'')=8$ and $(G'', \sigma'')$ can only be one of $\widehat{\Gamma}_3, \widehat{\Gamma}_4$, and $\widehat{\Gamma}_5$. However,  $(G'',\sigma'')$ has a $2$-vertex $y$,  contradicting the fact that $\widehat{\Gamma}_3, \widehat{\Gamma}_4$, and $\widehat{\Gamma}_5$ are all cubic.
\end{proof}

\begin{lemma}\label{lem:K42sub}
    The minimum counterexample $(G,\sigma)$ contain no $\widehat H_1$ as a subgraph, where $\widehat H_1$ is the graph depicted in~\Cref{fig:k4sub}.
\end{lemma}

\begin{proof}
    Assume for the contradiction that $(G,\sigma)$ contains $\widehat H_1$ as a subgraph. We form $(G',\sigma')$ from $(G, \sigma)$ by replacing $\widehat H_1$ with the subgraph in~\Cref{fig:c3-2}. it is easy to see that $(G', \sigma')$ is $2$-edge-connected and simple. Moreover, we claim that $F(G', \sigma')=F(G, \sigma)-1$. Assume not and it implies that $(G', \sigma')$ contains an unequilibrated cut $C:=[A,B]$. We construct a new cut $C'$ (of $(G, \sigma)$) as follows:

 \begin{equation*}
	C'=\left\lbrace  
	\begin{array}{ll}
		[A,B\cup\{a_i,b_i,c_i,d_i\}] &\text{ if }x_i\in A,\{a_i,y_i\}\subset B,\\
        
        [A\cup\{c_i\},B\cup\{b_i,d_i\}] &\text{ if }\{x_i,a_i\}\subset A,y_i\in B,\\
        
		[A\cup\{a_i,b_i,c_i,d_i\},B\setminus \{a_i\}] &\text{ if }\{x_i,y_i\}\subset A, a_i\in B, \\

        [A\cup\{b_i,c_i,d_i\},B] &\text{ if }\{x_i,a_i,y_i\}\subset A,\\
	\end{array}
	\right.
    \end{equation*}
   where $C'=C\setminus \{x_iy_i\}\cup\{x_ic_i\}$ in the first case, $C'=C\setminus \{a_iy_i,x_iy_i\}\cup\{a_ib_i,a_id_i,c_ib_i,c_id_i\}$ in the second case, $C'=C\setminus \{x_ia_i,a_iy_i\}$ in the third case, and $C'=C$ in the last case. It is easy to see that $C'$ is an unequilibrated cut of $(G, \sigma)$, a contradiction to~\Cref{lem:bad-cut}. Hence, we have that $$F(G',\sigma')=F(G, \sigma)-1>\frac{1}{3}v(G)-1=\frac{1}{3}(v(G')+3)-1=\frac{1}{3}v(G').$$
Note that $v(G')=v(G)-3\geq 7$. If $(G', \sigma')$ is not isomorphic to any of $\widehat{\Gamma}_i$'s, then it is a smaller counterexample, contradicting the minimality of $(G, \sigma)$. So we know that $v(G')=8$ and 
$(G', \sigma')$ can only be one of $\widehat{\Gamma}_3, \widehat{\Gamma}_4$, and $\widehat{\Gamma}_5$. It is impossible since   $(G',\sigma')$ has a $2$-vertex $a_i$.
\end{proof}

\begin{lemma}\label{lem:H4}
The minimum counterexample $(G, \sigma)$ contains no $\widehat H_2$ as a signed subgraph, where $\widehat H_2$ is the graph depicted in~\Cref{fig:widehat H2} and $x\not\in \{b',c'\}$.    
\end{lemma}

\begin{figure}[htbp]
\begin{subfigure}[t]{.48\textwidth}
		\centering
	     \begin{tikzpicture}[>=latex,
		roundnode/.style={circle, draw=black!90, thick, minimum size=2mm, inner sep=0pt},
        squarenode/.style={rectangle, draw=black!90, thick, minimum size=2mm, inner sep=0pt},
        scale=0.8
		]
          \node[roundnode] (x) at (-1,0) {};
          \node[roundnode] (u) at (0,0) {};
          \node[roundnode] (y) at (1,0.8) {};
          \node[roundnode] (z) at (1,-0.8) {};
          \node[roundnode] (b) at (2,0.8) {};
          \node[roundnode] (b') at (3,0.8) {};
          \node[roundnode] (c) at (2,-0.8) {};
          \node[roundnode] (c') at (3,-0.8) {};
        
          \draw[fill=white,line width=0.2pt] (-1,0) node[below=1mm] {$x$};
          \draw[fill=white,line width=0.2pt] (0,0) node[below=1mm] {$u$};
          \draw[fill=white,line width=0.2pt] (1,0.8) node[above=1mm] {$y$};
          \draw[fill=white,line width=0.2pt] (1,-0.8) node[below=1mm] {$z$};
          \draw[fill=white,line width=0.2pt] (2,0.8) node[above=1mm] {$b$};
          \draw[fill=white,line width=0.2pt] (3,0.8) node[above=1mm] {$b'$};
          \draw[fill=white,line width=0.2pt] (2,-0.8) node[below=1mm] {$c$};
          \draw[fill=white,line width=0.2pt] (3,-0.8) node[below=1mm] {$c'$};

          \draw[line width=1pt, blue] (b')--(b)--(z)--(u)--(y)--(c)--(c');
          \draw[line width=1pt, blue] (u)--(x);
          \draw[densely dotted, line width=1pt, red] (b)--(y);
          \draw[densely dotted, line width=1pt, red] (c)--(z);
	    \end{tikzpicture} 
        \caption{A signed subgraph $\widehat H_2$}
		\label{fig:widehat H2}
  \end{subfigure}
  \begin{subfigure}[t]{.48\textwidth}
  \centering
		\begin{tikzpicture}[>=latex,
		roundnode/.style={circle, draw=black!90, thick, minimum size=2mm, inner sep=0pt},
        squarenode/.style={rectangle, draw=black!90, thick, minimum size=2mm, inner sep=0pt},
        scale=0.8
		]
          \node[roundnode] (x) at (-1,0) {};
          \node[roundnode] (u) at (0,0) {};
          \node[roundnode] (b') at (3,0.8) {};
          \node[roundnode] (c') at (3,-0.8) {};
        
          \draw[fill=white,line width=0.2pt] (-1,0) node[below=1mm] {$x$};
          \draw[fill=white,line width=0.2pt] (0,0) node[below=1mm] {$u$};
          \draw[fill=white,line width=0.2pt] (3,0.8) node[above=1mm] {$b'$};
          \draw[fill=white,line width=0.2pt] (3,-0.8) node[below=1mm] {$c'$};
          
          \draw[line width=1pt, blue] (c')--(u)--(x);
          \draw[densely dotted, line width=1pt, red] (u)--(b');
	    \end{tikzpicture} 
        \caption{$(G',\sigma')$}
		\label{fig:H4'}
    \end{subfigure}
    \caption{Configurations in~\Cref{lem:H4}} 
    \label{fig:H4-subgraph}
\end{figure}  

\begin{proof}
Assume to the contrary that $(G, \sigma)$ contains $\widehat H_2$ as a subgraph, labeled as in~\Cref{fig:widehat H2}. Note that $b'\ne c'$, as otherwise $(G, \sigma)$ contains $\widehat H_1$ as a subgraph, contradicting~\Cref{lem:K42sub}. We form a signed graph $(G', \sigma')$ from $(G, \sigma)$ by deleting vertices $\{y,z,b,c\}$, and adding a negative edge $ub'$ and a positive edge $uc'$, as shown in~\Cref{fig:H4'}. It is easy to see that $(G', \sigma')$ is $2$-edge-connected. Since $x\not\in \{b', c'\}$, $(G', \sigma')$ contains no parallel edges. Moreover, we claim that $F(G', \sigma')=F(G, \sigma)-1$. Assume not and it implies that $(G', \sigma')$ contains an unequilibrated cut $C:=[A,B]$. We consider the following cut $C'$ of $(G, \sigma)$:
 \begin{equation*}
	C'=\left\lbrace  
	\begin{array}{ll}
		[A\cup\{y,z\},B\cup\{b,c\}] &\text{ if }u\in A,\{b',c'\}\subset B,\\
        
        [A\cup\{y,b,z,c\},B] &\text{ if }\{u,b'\}\subset A,c'\in B,\\
        
		[A\cup\{y,c\},B\cup\{b,z\}] &\text{ if } \{u,c'\}\subset A,b'\in B, \\

        [A\cup\{y,b,z,c\},B] &\text{ if }\{u,b',c'\}\subset A,\\
	\end{array}
	\right.
    \end{equation*}
    where $C'=C\setminus \{b'u,c'u\}\cup\{by,bz,cy,cz\}$ in the first case, $C'=C\setminus \{c'u\}\cup\{c'c\}$ in the second case, $C'=C\setminus \{b'u\}\cup\{by,uz,cz\}$ in the third case, and $C'=C$ in the last case. The edge-cut $C'$ is an unequilibrated cut of $(G, \sigma)$, a contradiction to~\Cref{lem:bad-cut}. Hence, we have that $$F(G',\sigma')=F(G, \sigma)-1>\frac{1}{3}v(G)-1=\frac{1}{3}(v(G')+4)-1> \frac{1}{3}v(G').$$ Note that $v(G')=v(G)-4\geq 6$. If $(G', \sigma')$ is not isomorphic to any of $\widehat{\Gamma}_i$'s, then it is a smaller counterexample, contradicting the minimality of $(G, \sigma)$. So we know that $v(G')=8$ and $(G', \sigma')$ can only be one of $\widehat{\Gamma}_3, \widehat{\Gamma}_4$, and $\widehat{\Gamma}_5$. Thus $(G,\sigma)$ is a signed subcubic graph on $12$ vertices with $3$ negative edges, i.e., $F(G, \sigma)=\frac{1}{3}v(G)$, a contradiction.
\end{proof}

\subsection[ProofOfLemma]{Proof of~\Cref*{lem:key}}\label{sec:proof of lemma}

Recall that for $i\in \{0,1,2,3\}$, $X_i=\{u\in X\mid d_Y(u)=i\}$, and for $j\in \{0,1,2\}$,
$Y_j=\{v\in Y\mid d_X(v)=j\}$. We aim to show that $$|X_3|+|Y_0|\leq |Y_2|.$$

It follows from \Cref{lem:cubic} that the minimum counterexample $(G, \sigma)$ is cubic, and from \Cref{lem:c3negative} that it contains no negative triangle.

\begin{lemma}\label{lem:Y0}
$Y_0=\emptyset$.
\end{lemma}

\begin{proof}
Assume to the contrary that there is a vertex $u\in Y_0$. Let $x$ and $y$ be two positive neighbors of $u$. Since $d_X(u)=0$, both $x$ and $y$ are in $Y$. If $xy$ is not a negative edge, then $G[\{u,x,y\}]\in \mathcal{T}$ is an induced subgraph of $G$, a contradiction. So $xy$ must be a negative edge. It implies that $u$ is in a negative triangle, a contradiction to~\Cref{lem:c3negative}.
\end{proof}

\begin{lemma}\label{lem:Vertex-A1}
Every vertex of $X_3$ has at least two neighbors that are vertices in $Y_2$.
\end{lemma}

  \begin{figure}[htbp]
  \centering
     \begin{tikzpicture}[>=latex,
		roundnode/.style={circle, draw=black!90, thick, minimum size=2mm, inner sep=0pt},
        scale=0.8]
          \node[roundnode] (u1) at (3.5,0.8) {};
          \node[roundnode] (x1) at (2,0) {};
          \node[roundnode] (x11) at (1,0) {};
          \node[roundnode] (y1) at (3.5,0) {};
          \node[roundnode] (z1) at (5,0) {};
          \node[roundnode] (a1) at (2,-0.8) {};
          \node[roundnode] (b1) at (3.5,-0.8) {};
          \node[roundnode] (c1) at (5,-0.8) {};

           \draw[fill=white,line width=0.2pt] (3.5,0.8) node[above=1mm] {$u$};
           \draw[fill=white,line width=0.2pt] (2,0) node[right=1mm] {$x$};
           \draw[fill=white,line width=0.2pt] (1,0) node[above=1mm] {$x'$};
           \draw[fill=white,line width=0.2pt] (5,0) node[right=1mm] {$z$};
           \draw[fill=white,line width=0.2pt] (3.5,0) node[right=1mm] {$y$};
           \draw[fill=white,line width=0.2pt] (2,-0.8) node[right=1mm] {$a$};
           \draw[fill=white,line width=0.2pt] (3.5,-0.8) node[right=1mm] {$b$};
           \draw[fill=white,line width=0.2pt] (5,-0.8) node[right=1mm] {$c$};

           \draw[line width=1pt, blue] (x11)--(x1)--(u1)--(z1);
          \draw[line width=1pt, blue] (u1)--(y1);
          \draw[densely dotted, line width=1pt, red] (x1)--(a1);
          \draw[densely dotted, line width=1pt, red] (y1)--(b1);
          \draw[densely dotted, line width=1pt, red] (z1)--(c1);
		\end{tikzpicture}
     \caption{A vertex $u\in X_3$ has three neighbors $x,y,z\in Y$}
    \label{fig:Vertex-A1}
\end{figure}
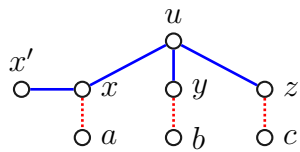

\begin{proof}
Let $u$ be an vertex of $X_3$ in $(G, \sigma)$ with its three neighbors $x,y,$ and $z$. Since $u\in X_3$, $\{x,y,z\}\subset Y$. We first claim that $\{x,y,z\}$ is an independent set, as shown in~\Cref{fig:Vertex-A1}. If not, then without loss of generality we may assume that $x$ is adjacent to $y$. By~\Cref{lem:c3negative}, $xy$ must be a positive edge (i.e., $uxy$ is a positive triangle). Now $G[\{u,x,y\}]\in \mathcal{C}$ is an induced subgraph of $G$, a contradiction.

Assume to the contrary that at least two of $x,y,$ and $z$ are not in $Y_2$. By~\Cref{lem:Y0}, we know that $Y_0=\emptyset$. So without loss of generality, we may assume $\{x, y\}\subset Y_1$, i.e., $d_X(x)=1$ and $d_X(y)=1$.  Let $a,b,c$ be the other neighbor of $x,y,z$ where $ax, by, cz$ are all negative edges in $(G, \sigma)$. By definition, each of $a$, $b$, and $c$ is in $Y$. Let $x'$ be the third neighbor of $x$. As $d_X(x)=1$ and $u\in X\cap N(x)$, we have that $x'\in Y$, i.e., $x'$ is incident with a negative edge. Thus $G[\{x,x',u,y,z\}]\in \mathcal{T}$ is an induced subgraph, unless $x'$ coincides with exactly one of $b$ and $c$. Without loss of generality, we assume that $x'=b$, i.e., $xb$ is a positive edge. Similar arguments apply to $y$, and hence $y$ is connected to exactly one of $a$ and $c$ via a positive edge. We consider the following two cases:

\begin{itemize}
    \item Assume that $yc\in E(G)$. In this case, $b$ is not the neighbor of $z$ or $c$, as otherwise either $G[\{x,b,z\}]\in \mathcal{T}$ or $G[\{x,b,c\}]\in \mathcal{T}$ is an induced subgraph of $(G, \sigma)$, a contradiction to~\Cref{lem:tree}. Moreover, since $x$ is neither a neighbor of $c$ nor $z$, the $6$-cycle $xbyczu$ has only one chord $yu$. By switching at an equilibrated cut $[y,c]$, the resulting signed graph $(G,\sigma')$ satisfies that $|E^-_{(G,\sigma')}|=F(G,\sigma)$. Furthermore, $(G,\sigma')$ contains two adjacent negative $4$-cycles sharing a common negative edge $yu$ (where the $6$-cycle is exactly $xbyczu$), a contradiction to~\Cref{lem:4435}\ref{lem:44}.

    \item Assume that $ya\in E(G)$. In this case, if $z$ is adjacent to exactly one of $a$ and $b$ via a positive edge, by symmetry, say $zb\in E(G)$, then $G[\{x, b, z\}]\in \mathcal{T}$ is an induced subgraph of $G$, a contradiction. Thus it contains $\widehat H_2$ as a subgraph, which contradicts~\Cref{lem:H4}. 
\end{itemize}
We complete the proof of the lemma.
\end{proof}

Now consider the bipartite subgraph of $G$ between $X_3$ and $Y_2$. By double counting on the number of edges between $X_3$ and $Y_2$, we have that $$\sum_{v\in X_3} d_{Y_2}(v)=\sum_{u\in Y_2} d_{X_3}(u).$$ For $v\in X_3$, by~\Cref{lem:Vertex-A1}, we have that $d_{Y_2}(v)\ge 2$; and for $u\in Y_2$, by the definition of $Y_2$, we have that $d_{X_3}(u)\leq d_X(u)= 2$. Therefore, since $Y_0=\emptyset$ by~\Cref{lem:Y0}, $$\left | X_3 \right |+\left | Y_0 \right |=\left | X_3 \right |\le \frac{1}{2}  \sum_{v\in X_3} d_{Y_2}(v)=\frac{1}{2} \sum_{u\in Y_2} d_{X_3}(u)\leq |Y_2|.$$ This completes the proof of \Cref{lem:key}.

\section*{Acknowledgments}
Jiaao Li is partially supported by National Key Research and Development Program of China (No. 2022YFA1006400), National Natural Science Foundation of China (Nos. 12571371, 12222108), Natural Science Foundation of Tianjin (No. 24JCJQJC00130), and the Fundamental Research Funds for the Central Universities, Nankai University. Zhouningxin Wang is partially supported by National Natural Science Foundation of China (No. 12301444) and the Fundamental Research Funds for the Central Universities, Nankai University.

\bibliographystyle{amsplain}

\begin{thebibliography}{99}

\bibitem{AACE1981}
J. Akiyama, D. Avis, V. Chv\'{a}tal, and H. Era.
Balancing signed graphs.
\emph{Discrete Appl. Math.}, Volume 3, Issue 4, 1981, 227--233.

\bibitem{B1983}
F. Barahona.
The max-cut problem on graphs not contractible to $K_5$.
\emph{Oper. Res. Lett.}, Volume 2, Issue 3, 1983, 107--111.

\bibitem{BL1986}
J. A. Bondy and S. C. Locke.
Largest bipartite subgraphs in triangle-free graphs with maximum degree three.
\emph{J. Graph Theory}, Volume 10, Issue 4, 1986, 477--504.

\bibitem{BCCS1977}
F. C. Bussema, S. \v{C}obelji\'{c}, D. M. Cvetkovi\'{c}, and J. J. Seidel.
Cubic graphs on $\leq 14$ vertices.
\emph{J. Combin. Theory Ser. B}, 23(2--3), 1977, 234--235.

\bibitem{CNSW2025}
C. Cappello, R. Naserasr, E. Steffen, and Z. Wang.
Critically $3$-frustrated signed graphs.
\emph{Discrete Math.}, Volume 348, Issue 1, 2025, 114258.

\bibitem{CS2022}
C. Cappello and E. Steffen.
Frustration-critical signed graphs.
\emph{Discrete Appl. Math.}, Volume 322, 2022, 183--193.

\bibitem{E1973}
C. S. Edwards.
Some extremal properties of bipartite subgraphs.
\emph{Canad. J. Math.}, Volume 25, Issue 3, 1973, 475--485.

\bibitem{E1975}
C. S. Edwards.
An improved lower bound for the number of edges in a largest bipartite subgraph.
Recent advances in graph theory (Proc. Second Czechoslovak Sympos., Prague, 1974), Academia, Prague, 1975, 167--181.

\bibitem{GW1995}
M. X. Goemans and D. P. Williamson.
Improved approximation algorithms for maximum cut and satisfiability problems using semidefinite programming.
\emph{J. ACM}, Volume 42, Issue 6, 1995, 1115--1145.

\bibitem{H1975}
F. Hadlock.
Finding a maximum cut of a planar graph in polynomial time.
\emph{SIAM J. Comput.}, Volume 4, Issue 3, 1975, 221--225.

\bibitem{HS1982}
G. Hopkins and W. Staton.
Extremal bipartite subgraphs of cubic triangle-free graphs.
\emph{J. Graph Theory}, Volume 6, Issue 2, 1982, 115--121.

\bibitem{K1972}
R. M. Karp.
Reducibility among Combinatorial Problems.
In: Miller, R. E., Thatcher, J. W., Bohlinger, J. D. (eds) \emph{IBM Res. Symp. Ser.}, Springer, Boston, MA, 1972, 85--103.

\bibitem{K1997}
J. Koml\'{o}s.
Covering odd cycles.
\emph{Combinatorica}, Volume 17, 1997, 393--400.

\bibitem{KV2004}
D. Kr\'{a}l and H. J. Voss.
Edge-disjoint odd cycles in planar graphs.
\emph{J. Combin. Theory Ser. B}, Volume 90, Issue 1, 2004, 107--120.

\bibitem{L1982}
S. C. Locke.
Maximum $k$-colorable subgraphs.
\emph{J. Graph Theory}, Volume 6, Issue 2, 1982, 123--132.

\bibitem{PT1982}
S. Poljak and D. Turz\'{i}k.
A Polynomial Algorithm for Constructing a Large Bipartite Subgraph, with an Application to a Satisfiability Problem.
\emph{Canad. J. Math.}, Volume 34, Issue 3, 1982, 519--524.

\bibitem{SZ1994}
P. Sol\'{e} and T. Zaslavsky. 
A coding approach to signed graphs. 
\emph{SIAM J. Discrete Math.}, Volume 7, Issue 4, 1994, 544--553.

\bibitem{XY2008}
B. Xu and X. Yu.
Triangle-free subcubic graphs with minimum bipartite density.
\emph{J. Combin. Theory Ser. B}, Volume 98, Issue 3, 2008, 516--524.

\bibitem{Z1998}
T. Zaslavsky.
A Mathematical Bibliography of Signed and Gain Graphs and Allied Areas.
\emph{Electron. J. Combin.}, Dynamic survey 8, 1998, 1--124.

\bibitem{Z2009a}
X. Zhu.
Bipartite density of triangle-free subcubic graphs.
\emph{Discrete Appl. Math.}, Volume 157, Issue 4, 2009, 710--714.

\bibitem{Z2009b}
X. Zhu.
Bipartite subgraphs of triangle-free subcubic graphs.
\emph{J. Combin. Theory Ser. B}, Volume 99, Issue 1, 2009, 62--83.

\bibitem{Z1990}
O. Z\'{y}ka.
On the bipartite density of regular graphs with large girth.
\emph{J. Graph Theory}, Volume 14, Issue 6, 1990, 631--634.
\end{thebibliography}

\end{document}